\documentclass{article}

\usepackage{amsmath,amsthm,amsfonts,amssymb,bm,tikz,float,booktabs,appendix}
\usepackage{booktabs,geometry,graphicx,multirow,enumerate,cleveref,array}
\usepackage[numbers,sort&compress]{natbib}
\usepackage[all]{xy}

\newtheorem{theorem}{Theorem}[section]
\newtheorem{lemma}{Lemma}[section]

\newtheorem{example}{Example}[section]

\numberwithin{equation}{section}
\numberwithin{figure}{section}
\numberwithin{table}{section}

\newcommand{\jp}{j+\frac{1}{2}}
\newcommand{\jm}{j-\frac{1}{2}}
\newcommand{\ljp}{|_{j+\frac{1}{2}}}
\newcommand{\ljm}{|_{j-\frac{1}{2}}}

\newcommand{\xjm}{x_{j-\frac{1}{2}}}
\newcommand{\xjp}{x_{j+\frac{1}{2}}}
\newcommand{\xjmm}{x^+_{j-\frac{1}{2}}}
\newcommand{\xjpp}{x^-_{j+\frac{1}{2}}}
\newcommand{\ff}{\frac{1}{2}}	
\newcommand{\bhj}{\bar{h}_j}

\newcommand{\thet}{{(\theta)}}
\newcommand{\tthet}{(\tilde{\theta})}
\newcommand{\tiltheta}{\tilde{\theta}}

\newcommand{\chtkl}{Ch^{2k+2\ell+2}} 
\newcommand{\chkl}{Ch^{k+\ell+1}}    

\newcommand{\tilth}{\tilde{\theta}}

\newcommand{\tilsig}{\tilde{\sigma}}

\newcommand{\la}{{(\lambda)}}
\newcommand{\tla}{(\tilde{\lambda})}
\newcommand{\tlam}{\tilde{\lambda}}

\newcommand{\sumjn}{\sum_{j =1}^N}

\newcommand{\z}[1]{\mathbb{Z}_{#1}}

\newcommand{\bb}[1]{\mathcal{#1}}
\newcommand{\hjop}[3]{\mathcal{H}_j^{ #1 }\left( #2,#3\right) }
\newcommand{\hop}[3]{\mathcal{H}^{ #1 }\left( #2,#3\right) }

\newcommand{\di}{D^{-1}_x}
\newcommand{\il}[1]{{#1}^{\ell}_I}

\newcommand{\wl}[1]{W_{#1}^{\ell}}
\newcommand{\wll}[1]{w_{#1}^{\ell}}
\newcommand{\wllz}[1]{w_{#1}^{\ell}(0)}

\newcommand{\wi}[1]{w_{#1}^i }
\newcommand{\wimo}[1]{w_{#1}^{i-1}}
\newcommand{\wipo}[1]{w_{#1}^{i+1}}
\newcommand{\wice}[1]{\left.w_{#1}^i\right|_{\ce}}
\newcommand{\wz}[1]{w_{#1}^{0}}
\newcommand{\wo}[1]{w_{#1}^{1}}

\newcommand{\ih}{\mathcal{I}_h}
\newcommand{\be}[1]{\bar{e}_{#1}}
\newcommand{\bez}[1]{\bar{e}_{#1}(0)}

\newcommand{\eut}{(e_{u})_t}                   
\newcommand{\epsilonut}{{(\epsilon_u)}_t}      
\newcommand{\betz}[1]{(\bar{e}_{#1})_t(0)}     
\newcommand{\bet}[1]{(\bar{e}_{#1})_t}         
\newcommand{\bexz}[1]{(\bar{e}_{#1})_x(0)}     
\newcommand{\wllt}[1]{\partial_t w_{#1}^{\ell}}

\newcommand{\pnt}{\partial_t^n}
\newcommand{\lenj}[1]{L_{j,#1}}

\newcommand{\intid}[1]{\int_{I} #1 \mathrm{d}x}
\newcommand{\intijd}[1]{\int_{\ce} #1 \mathrm{d}x}

\newcommand{\leri}[1]{\left( #1 \right)}

\newcommand{\inttdtau}[1]{\int_{0}^{t} #1 \mathrm{d}\tau}

\newcommand{\intcd}[1]{\int_{\ce} #1 \mathrm{d}x }
\newcommand{\fd}{\frac{\mathrm{d}}{\mathrm{d}t}}

\newcommand{\norm}[1]{\Vert #1 \Vert}
\newcommand{\norms}[1]{\Vert #1 \Vert^2}

\newcommand{\normice}[1]{\Vert #1 \Vert_{\infty,I_j}}

\newcommand{\normki}[1]{\Vert #1 \Vert_{k+1,\infty,I_j}}

\newcommand{\jump}[1]{\left[\kern-0.15em\left[ #1 \right]\kern-0.15em\right]}
\newcommand{\jumpbig}[1]{\big[\kern-0.3em\big[ #1 \big]\kern-0.3em\big]}
\newcommand{\jumpmo}[1]{\left|\kern-0.15em\left[ #1 \right]\kern-0.15em\right|}
\newcommand{\jumpmos}[1]{\left|\kern-0.15em\left[ #1 \right]\kern-0.15em\right|^2}
\newcommand{\jumpjp}[1]{{\left[\kern-0.15em\left[ #1 \right]\kern-0.15em\right]}_{\jm}}
\newcommand{\jumpsjp}[1]{{\left[\kern-0.15em\left[ #1 \right]\kern-0.15em\right]}^2_{\jm}}

\newcommand{\ave}[1]{\left\{\kern-0.3em\left\{ #1 \right\}\kern-0.3em\right\}}
\newcommand{\vertce}{\vert_{\ce}}
\newcommand{\ce}{I_j}

\newcommand{\mo}[1]{\left| #1 \right|}	
\newcommand{\mos}[1]{\left| #1 \right|^2}

\newcommand{\red}[1]{\textcolor{red}{#1}}

\newcommand{\pt}{P_{\theta}}
\newcommand{\tpt}{\tilde{P}_{\theta}}
\newcommand{\ptt}{P_{\tilde{\theta}}}
\newcommand{\tptt}{\tilde{P}_{\tilde{\theta}}}
\newcommand{\pl}{P_{\lambda}}
\newcommand{\tpl}{\tilde{P}_{\lambda}}
\newcommand{\ptl}{P_{\tilde{\lambda}}}
\newcommand{\tptl}{\tilde{P}_{\tilde{\lambda}}}

\newcommand{\psig}{P_{\sigma}}

\newcommand{\nnum}[1]{\mathrm{\uppercase\expandafter{\romannumeral#1}}}
\newcommand{\ind}{\quad~\!}
\newcommand{\rrom}[1]{\rm(\uppercase\expandafter{\romannumeral  #1}) }

\allowdisplaybreaks

\begin{document}

\title{Superconvergence of the local discontinuous Galerkin method with generalized numerical fluxes for one-dimensional linear time-dependent fourth-order equations}
\vspace{0.035in}
\date{}

\author{
Linhui Li\footnote{School of Mathematics, Harbin Institute of Technology, Harbin {\rm 150001}, Heilongjiang, China.},  Xiong Meng\footnote{Corresponding author. School of Mathematics, Harbin Institute of Technology, Harbin {\rm 150001}, Heilongjiang, China. The research of this author was supported by National Natural Science Foundation of China grants 12371365, 11971132, Natural Science Foundation of Heilongjiang Province grant  YQ2021A002, and the Fundamental Research Funds for the Central Universities grant HIT.OCEF.2022031.},
and Boying Wu\footnote{School of Mathematics, Harbin Institute of Technology, Harbin {\rm 150001}, Heilongjiang, China. The research of this author was supported by National Natural Science Foundation of China grant 11971131.}
}

 \maketitle

\begin{abstract}
In this paper, we concentrate on the superconvergence of the local discontinuous Galerkin method with generalized numerical fluxes for one-dimensional linear time-dependent fourth-order equations. The adjustable numerical viscosity of the generalized numerical fluxes is beneficial for long time simulations with a slower error growth. By using generalized Gauss--Radau projections and correction functions together with a suitable numerical initial condition, we derive, for polynomials of degree $k$, $(2k+1)$th order superconvergence for the numerical flux and cell averages, $(k+2)$th order superconvergence at generalized Radau points, and $(k+1)$th order for error derivative at generalized Radau points. Moreover, a supercloseness result of order $(k+2)$ is established between the generalized Gauss--Radau projection and the numerical solution.
Superconvergence analysis of mixed boundary conditions is also given. Equations with discontinuous initial condition and nonlinear convection term are numerically investigated, illustrating that the conclusions are valid for more general cases.
\end{abstract}



\section{Introduction}
In this paper, we investigate superconvergence of local discontinuous Galerkin (LDG) methods with generalized numerical fluxes for one-dimensional linear fourth-order problem
\begin{subequations}\label{eq1}
\begin{align}
u_t +  \alpha u_x + \beta u_{xx} + u_{xxxx} &= 0,~& &(x,t)\in I\times (0,T],\label{eq1-a}\\
u(x,0) &= u_0(x),~& &x\in I, \label{eq1-b}
\end{align}
\end{subequations}
where $\alpha$ and $\beta$ are constant, and $I = (0, 2 \pi)$. Periodic and mixed boundary conditions are considered. Note that, in  \eqref{eq1-a}, $u_{xxxx}$ dominates in spite of the anti-diffusion term $\beta u_{xx}$ with $\beta>0$. The generalized numerical fluxes with flexible numerical viscosities allow us to obtain a slower error growth for long time simulations, when compared with the LDG scheme using the upwind and alternating fluxes; see Figure \ref{fig:fig1label} below. With the help of correction functions and an elaborate numerical initial condition, by establishing a superconvergent bound for interpolation errors, supercongence for the numerical flux, cell averages, Radau points as well as supercloseness are derived.

The discontinuous Galerkin (DG) method, allowing discontinuities across cell boundaries in the finite element space, was proposed mainly for solving hyperbolic conservation laws and systems \cite{cockburn1989TVB2,cockburn1990RKDG4,cockburn1998RKDG5}.
In \cite{cockburn1998LDG}, the LDG method was developed for solving  convection-diffusion equations, which was achieved by introducing an auxiliary variable and rewriting the original problem into a first-order system to which the DG method can be applied.
Later, the LDG methods have been widely adopted to solve high-order partial differential equations (PDEs), such as Korteweg--de Vries (KdV) type equations \cite{yan2002local},  Schr\"{o}dinger equations \cite{xu2005local2}, the Zakharov--Kuznetsov equation \cite{xu2005local}, and viscous Burgers equations \cite{hutridurga2023vvb}. For more details of DG and LDG methods, we refer to the review paper \cite{shu2016discontinuous}.

The fourth-order PDEs have numerous physical and engineering applications. For example, fourth-order boundary value problems can describe the bending of an elastic beam, and the Cahn--Hilliard equation reflects the process of phase separation, in which the properties of fluid thermodynamics transfer smoothly  from one phase to another \cite{cahn1958hilliard}. DG and LDG methods have been studied for solving fourth-order PDEs. In \cite{dong2009analysis}, Dong and Shu used the LDG method for fourth-order time-dependent problems to obtain optimal error estimates in one- and multi-dimensional spaces.
A free-energy stable DG method for the Cahn--Hilliard equation with non-conforming elements was developed in \cite{ntoukas2021ch}.

Superconvergence of DG and LDG methods has gained more attention in recent years. Based on the correction function technique, in \cite{cao2014SCDG}, Cao et al. proved superconvergence of numerical fluxes, cell averages and Radau points of DG methods for linear hyperbolic equations. Superconvergence of LDG methods for high-order linear problems was given in \cite{cao2017highorder}. Superconvergence of ultraweak-LDG method for linear fourth-order equations can be found in \cite{liuyong2020SCofUWLDG4}. We would like to emphasize that superconvergence property is probably sensitive to the numerical initial condition and a special initial discretization should be chosen.
In addition, based on Fourier analysis, \cite{zhong2022sc} studied the superconvergence properties of various direct DG methods for diffusion equations, and presented quantitative errors at Lobatto points.

In the design of DG and LDG schemes, the choice of numerical fluxes plays an important role to guarantee stability and optimal order of accuracy. For linear hyperbolic equations, instead of using classical monotone or upwind fluxes, \cite{meng2016optimal} proposed a nonmonotone upwind-biased flux and showed $L^2$-stability as well as optimal error estimates, in which a global generalized Gauss--Radau (GGR) projection is constructed. Almost at the same time, the GGR projection is considered in LDG methods for solving the Burgers--Poisson equation in \cite{liu2015local}. The generalized alternating fluxes were then developed for LDG methods solving linear convection-diffusion equations \cite{cheng2017application}, in which a modified GGR projection is designed to deal with different weighs of generalized numerical fluxes. For the Vlasov-viscous Burgers system, a coupling DG and LDG method with generalized numerical fluxes was introduced in \cite{hutridurga2023vvb}, which is mass and momentum conservative. Moreover, superconvergence of DG methods with upwind-biased fluxes for linear hyperbolic equations and LDG methods with generalized alternating fluxes for linear convection-diffusion equations were given in \cite{cao2017upwind-biased} and \cite{liu2021sc}, respectively.

Generalized fluxes with flexible numerical viscosities may be useful for long time simulations, as shown in Figure \ref{fig:fig1label}. It would be interesting to investigate superconvergence of LDG methods for linear fourth-order problems, especially when generalized fluxes with different weights are concerned. The superconvergence analysis mainly involves two difficulties. One is that the correction function is globally coupled, which can be solved by the property of circulant matrices. The other difficulty is a proper choice of numerical initial condition, which is achieved by the exact collocation for the third order derivative. Superconvergent initial error estimates can be obtained by using the discrete Poincar\'{e} inequality and the relationship between the derivative a numerical solution and another auxiliary variable.

The paper is organized as follows.
In Section \ref{scheme}, we present the LDG scheme using generalized numerical fluxes for one-dimensional fourth-order problem with periodic boundary conditions and introduce some preliminaries regarding GGR projections as well as properties of DG operators. In Section \ref{main body}, we construct correction functions with superconvergence property and derive a superconvergent bound for interpolation errors. Section \ref{sc analysis} is the main body of the paper, in which a suitable numerical initial condition is chosen, and a supercloseness result between the interpolation function and the LDG solution is established, followed by  superconvergence of numerical fluxes, cell averages and Radau points. Extension to mixed boundary conditions is provided in Section \ref{extensions}. In
Section \ref{experiments}, numerical experiments with linear equations as well as nonperiodic boundary conditions are given to validate theoretical results, and problems with discontinuous initial condition and nonlinear convection term are presented to illustrate that the superconvergent results hold for more general cases. We end in Section \ref{conclusion} with conclusions and perspectives for future work.

\section{The LDG scheme and preliminaries}\label{scheme}
Without loss of generality, consider $ \alpha = \beta = 1 $ in \eqref{eq1}, i.e.,
\begin{subequations}\label{eq_lin_4_}
\begin{align}
u_t +  u_x + u_{xx} + u_{xxxx} &= 0,~& &(x,t)\in I\times (0,T], \label{eq_lin_4_-a}\\
u(x,0) &= u_0(x),~& &x\in I, \label{eq_lin_4_-b}
\end{align}
\end{subequations}
where $I=(0, 2\pi)$. To clearly display superconvergence analysis of LDG methods with generalized numerical fluxes, we mainly consider periodic boundary conditions, and the case with mixed boundary conditions is discussed in Section \ref{extensions}.

\subsection{The LDG scheme}
We adopt the following standard notation.
Consider a partition $\mathcal{I}_h = \big\{\ce=(\xjm ,\xjp)\big\}^N_{j=1},~ j\in\z{N}$,
where, for any positive integer $\ell$, $\z{\ell} = \{1,\ldots,\ell\}$.
The cell center and cell length are denoted by
$x_j = \ff(\xjm +\xjp)$  and $h_j = \xjp-\xjm$, respectively.
Denote $h=\max_{j}h_j$, $\bhj =\frac{h_j}{2}$, and use $\Gamma_{\!h}$ to represent the set of cell boundary points. Assume that the mesh is quasi uniform, i.e., there is a positive constant $\gamma$ such that $h_j\ge \gamma h$, $\forall j\in \z{N}$.
The DG finite element space is
$$
V_{h}=\left\{v \in L^{2}(I):\left.v\right|_{I_{j}} \in P^k\left(I_{j}\right), \forall j \in \z{N}\right\} ,
$$
where $P^k(I_j)$ is the space of polynomials of degree at most $k$ in $I_j$.

For any integer $\ell \ge 0$, the Sobolev space in $D$ is denoted as $W^{\ell ,p}(D)$ equipped with the norm $\norm{\cdot}_{\ell ,p,D}=\norm{\cdot}_{W^{\ell ,p}(D)} $, and, when $p=2$,  $H^\ell(D) = W^{\ell ,2}(D)$ and $\norm{\cdot}_{\ell ,D} = \norm{\cdot}_{H^\ell(D)} $.
Here and below, the index $D$ or $\ell$ will be omitted when $D = I$ or $\ell=0$, and an unmarked norm is the standard $L^2$ norm in $I$.
For $p=2,\infty$, we denote the broken Sobolev spaces as
\begin{align*}
	W^{\ell,p}(\mathcal{I}_h) = \left\{v\in L^{2}(I) :\left.v\right|_{I_{j}} \in W^{\ell,p}(I_{j}), \forall j \in \z{N}\right\},
\end{align*}
with the norm
$$\norm{u}_{\ell} \triangleq \norm{u}_{H^\ell(\ih)} = \Big(\sum_{j=1}^N \norm{u}^2_{\ell,\ce}\Big)^{\frac{1}{2}},\quad \norm{u}_{\ell,\infty} \triangleq \norm{u}_{W^{\ell,\infty}(\ih)} = \max_{j\in\z{N}} \norm{u}_{\ell,\infty, \ce}.$$
Moreover, the $ L^2 $ norm for boundaries is
$
	\norm{u}_{\Gamma_{\!h}} = \leri{\sumjn\left[|u^+_{\jm}|^2 + |u
		^-_{\jp}|^2\right]}^{\ff},
$
and the seminorm
\begin{equation}\label{semi}
	\jumpmo{u} = \Big(\sumjn\jumpsjp{u}\Big)^{\ff}
\end{equation}
with $ \jump{u} = u^{+} - u^{-}$, where $u^{\pm}_{\jp}$ are limits from the right and left cells. $A\lesssim B$ means that $A$ can be bounded by $B$ multiplied by a positive constant independent of $h$.

To define the LDG scheme, we rewrite \eqref{eq_lin_4_-a} into a first-order system
\begin{align*}
u_t + ( u + p + r)_x = 0,\quad r - q_x = 0,\quad q - p_x = 0, \quad p - u_x = 0.
\end{align*}
Then, the semi-discrete LDG scheme for solving \eqref{eq_lin_4_} is to find  $u_h$, $p_h$, $q_h$, and $r_h\in V_h$ such that
\begin{subequations}\label{num_lin_4_LDG}
\begin{align}\label{LDG1}
		((u_h)_t,v)_j - \hjop{\theta}{u_h}{v} - \hjop{\tlam}{p_h}{v} - \hjop{\tiltheta}{r_h}{v} &= 0, \\ \label{LDG2}
		(r_h,\phi)_j + \hjop{\lambda}{q_h}{\phi} &= 0, \\ \label{LDG3}
		(q_h,\psi)_j + \hjop{\tlam}{p_h}{\psi} &= 0, \\
		(p_h,\zeta)_j + \hjop{\theta}{u_h}{\zeta} &= 0\label{LDG4}
\end{align}
\end{subequations}
hold for all $v$, $\phi$, $\psi$ and $\zeta\in V_h$ and $j\in \z{N}$. Here and below, $(\cdot, \cdot )_j$ denotes the $L^2 $ inner product in $\ce$ with $(\cdot, \cdot ) = \sumjn(\cdot, \cdot )_j$, and
\begin{equation}\label{dgo}
	\hjop{\alpha}{w}{v} = (w,v_x)_j - w^{(\alpha)}v^-{\ljp} + w^{(\alpha)}v^+{\ljm},\quad \hop{\alpha}{w}{v} = \sum_{j=1}^{N}\hjop{\alpha}{w}{v},\quad j\in \z{N}.
\end{equation}
The following generalized numerical fluxes with two weights $\theta, \lambda \neq \ff$ are chosen, namely
\begin{equation} \label{gflux_lin_4_LDG}
	 \hat{r}_h = r^{\tthet}_h, \quad  \hat{q}_h = q^{\la}_h,\quad \hat{p}_h = p^{\tla}_h,\quad \hat{u}_h = u^{\thet}_h,
\end{equation}
where, for $\sigma = \theta, \lambda$,
$u^{(\sigma)}_{\jp} = \sigma u^-_{\jp} + \tilde \sigma u^+_{\jp}$ and  $\tilde{\sigma} =  1-\sigma$.

\subsection{Preliminaries}
For $ u\in H^1(\ih)$, the GGR projection $\pt u$ \cite{cheng2017application} is defined as the unique function in $ V_h $ satisfying
\begin{subequations}\label{ggr_pro}
\begin{align}
		&~& &~& &~& &~&\intcd{\left(u-\pt u\right) v_h} &= 0, & &\forall v_h \in P^{k-1}(\ce),		&~& &~& &~& &~& \label{ggr_pro_1}\\
		&~& &~& &~& &~&\left(u-\pt u \right)^{(\theta)}_{\jp} &= 0,& &j\in \z{N},		&~& &~& &~& &~& \label{ggr_pro_2}
\end{align}
\end{subequations}
which has the following optimal approximation property
\begin{align}\label{pt error}
	\norm{u-\pt u}_{\ce} + h^{\ff}\norm{u-\pt u}_{\infty,\ce} \le Ch^{k+\frac{3}{2}}\norm{u}_{k+1,\infty},
\end{align}
where $C$ is independent of $h$.

Consider the Legendre expansion
\begin{align*}
		u(x,t) = \sum_{m=0}^{\infty}u_{j,m}(t)\lenj{m}(x),\quad u_{j,m}(t) = \frac{2m+1}{h_j}(u,\lenj{m})_j,
\end{align*}
where $L_{j,m}(x)$ denotes the rescaled Legendre polynomial of degree $m$ in $\ce$.
It follows from the orthogonality of $\pt u$ in \eqref{ggr_pro_1} that
\begin{align}\notag
		\pt u  (x,t) = \sum_{m=0}^{k}u_{j,m}(t)\lenj{m}(x) + \bar{u}_{j,k}(t)\lenj{k}(x),
\end{align}
and $\bar{u}_{j,k}$ can be determined by \eqref{ggr_pro_2}. Consequently,
\begin{align}\label{v-ptv}
		(u-\pt u ) (x,t) = -\bar{u}_{j,k}(t)\lenj{k}(x) + \sum_{m=k+1}^{\infty}u_{j,m}(t)\lenj{m}(x).
\end{align}
The Bramble--Hilbert lemma and scaling arguments yield
\begin{align}\label{vjk esti}
\mo{\bar{u}_{j,k}} \lesssim h^{k+1}\normki{u},
\end{align}
which will be used later in the superconvergence analysis of correction functions in Section \ref{correction function}.

In the construction of correction functions, the following integral operator is useful, as defined in \cite{cao2017upwind-biased}
\begin{equation}\label{dxi}
	D^{-1}_x v(x) = \frac{1}{\bar{h}_j}\int_{x_{\jm}}^{x}v(\tilde{x})\mathrm{d} \tilde{x},\quad \tilde{x} \in \ce.
\end{equation}
Clearly, $ v(x) = \bar{h}_j\leri{\di v(x)}_x $.
If $ v $ is taken as  $ L_{j,m}$, we have, by the recurrence relationships of Legendre polynomials,
\begin{align}\label{inv operator}
\di L_{j,m}(x) = \frac{1}{2m+1}\left(L_{j,m+1}-L_{j,m-1}\right)(x).
\end{align}

In the analysis of generalized fluxes with different weights, the following generalized skew-symmetry property is needed, as shown in \cite{lijia2020analysisKdV}. It reads, for $w,v\in H^1(\ih)$ and weights $\theta_1, \theta_2$,
\begin{align}\label{DG operator}
		\bb{H}^{\theta_1}(w,v) + \bb{H}^{\theta_2}(v,w) &= ( \tilde{\theta}_2  - \theta_1)\sum_{j=1}^N\jump{w}_{\jm}\jump{v}_{\jm}.
\end{align}
Also, the discrete Poincar\'{e} inequality in \cite{brenner2003poincare} is helpful in deriving superconvergent initial error estimates. That is, for $ \zeta \in H^1(\ih)$, there holds
\begin{align}\label{dp}
		\norms{\zeta} \le C\Big( \mos{\zeta}_{H^1(\ih)} + h^{-1}\jumpmos{\zeta} + \big|\intid{\zeta}\big| \Big),
\end{align}
where $ \mos{\zeta}_{H^1(\ih)} = \sumjn \mos{\zeta}_{H^1(\ce)} $, $\jumpmo{\zeta}$ has been defined in \eqref{semi} and $ C $ is independent of $ h$.
Moreover, the following inverse inequalities are necessary, namely for $ v\in V_h $
\begin{equation}\label{inv}
\norm{v_x} \le C h^{-1}\norm{v},\quad
\norm{ v}_{\Gamma_{\!h}} \le C h^{-\ff}\norm{v},\quad
\norm{ v}_{\infty} \le C h^{-\ff}\norm{v},
\end{equation}
where $C$ is independent of $h$.

Let us finish this section by showing a lemma concerning the relationship between the $L^2$ norm of derivative, the jump seminorm of LDG solutions and the $L^2$ norm of the adjacent auxiliary variable, when the LDG scheme with generalized numerical fluxes is considered. Using the same approach as that in \cite{wang2019implicit}, we have the following lemma.
\begin{lemma}\label{derivative jump}
Assume that $u_h, p_h, q_h, r_h$ are solutions to the LDG scheme \eqref{num_lin_4_LDG} with generalized numerical fluxes \eqref{gflux_lin_4_LDG}. Then, the following relationships hold
\begin{subequations}\label{l123}
\begin{align}
\norm{(q_h)_x} + h^{-\ff}\jumpmo{q_h} &\lesssim \norm{r_h},\label{l1} \\
\norm{(p_h)_x} + h^{-\ff}\jumpmo{p_h} &\lesssim \norm{q_h},\label{l2} \\
\norm{(u_h)_x} + h^{-\ff}\jumpmo{u_h} &\lesssim \norm{p_h}. \label{l3}
\end{align}
\end{subequations}
\end{lemma}

\section{Correction functions and interpolation functions}\label{main body}
To obtain superconvergence property, we construct correction functions in Section \ref{correction function}, and establish a superconvergent bound of interpolation errors in Section \ref{rhs}.
\subsection{Correction functions}\label{correction function}
We start by constructing a series of functions  $\wi{u}$, $\wi{p}$, $\wi{q}$ and $\wi{r} \in V_h$, $i\in \z{k}$ satisfying
\begin{subequations}\label{wi_lin_4}
	\begin{align}
		(\wi{u}-\bhj \di \wimo{p},\phi)_j &= 0,& &\left(\wi{u}\right)^{\thet}_{\jp} = 0,\label{wi_lin_4_1}\\
		(\wi{p}-\bhj \di \wimo{q},\phi)_j& = 0,& &\left(\wi{p}\right)^{\tla}_{\jp} = 0,\label{wi_lin_4_2}\\
		(\wi{q}-\bhj \di \wimo{r},\phi)_j& = 0,& &\left(\wi{q}\right)^{\la}_{\jp} = 0,
		\label{wi_lin_4_3}\\
		( \wi{u} +  \wi{p} + \wi{r}+\bhj \di \partial_t \wimo{u},\phi)_j &= 0,& & \left(\wi{r}\right)^{\tthet}_{\jp}  =0 \label{wi_lin_4_4}
	\end{align}
\end{subequations}
for $ \phi \in P^{k-1}(\ce) $ and $j\in \z{N}$, where
\begin{align}\label{wi initial}
	\wz{u} = u-\pt u,\quad \wz{p} = p-\ptl p,\quad \wz{q} = q-\pl q,\quad \wz{r} = r-\ptt r.
\end{align}
Then, for $\ell\in \z{k}$, the correction functions are defined by
\begin{align}\label{wl}
	\wl{v}  = \sum_{i =1}^{\ell} \wi{v}, ~~v = u, p, q, r.
\end{align}
Superconvergence of $\wi{v}$ is shown in the following lemma.

\begin{lemma}\label{pro cf}		
The functions $\wi{v}$, $v = u,p,q,r$, $ i\in \z{k} $ defined by \eqref{wi_lin_4}--\eqref{wi initial} are uniquely determined and satisfy
\begin{subequations}\label{pro cf fom}
\begin{align}
\norm{\partial^n_t \wi{v}}_{\infty} &\lesssim h^{k+i+1}\norm{\partial^n_t v}_{k+i+1,\infty},\quad n=0,1,\label{pro of correction function 1}\\
(\partial^n_t \wi{v},\psi)_j &= 0,\quad \forall \psi\in P^{k-i-1}(\ce), \quad j\in \z{N}. \label{pro of correction function 2}
\end{align}
\end{subequations}
\end{lemma}

\begin{proof}
We prove this lemma by induction. For $i\in \z{k}$, denote
\begin{subequations}\label{wi_rep}
\begin{align*}
	\wice{v} = \sum_{m=0}^{k} v^i_{j,m} \lenj{m}(x), \quad v = u, p, q, r,
\end{align*}
\end{subequations}
where $v^i_{j,m}$ are coefficients to be determined later.

\textbf{Step 1:} For $i = 1$, choosing $v = \lenj{m}$ ($j\in \z{N}$) in \eqref{wi_lin_4_1} with $m\le k-1$, and taking into account \eqref{wi initial} together with \eqref{inv operator}, we arrive at the following equality
\begin{align*}
\left(\wo{u} - \bhj \di \wz{p},\ \lenj{m}\right)_j  = \left(\sum_{m=0}^{k} u^1_{j,m} \lenj{m}  +  \frac{\bhj \bar{p}_{j,k}}{2k+1} \left(L_{j,k+1}-L_{j,k-1}\right), ~  \lenj{m}\right)_j = 0,
\end{align*}
where $ \bar{p}_{j,k} $ is the coefficient in \eqref{v-ptv} with $ u $, $ \theta $ replaced by $ p $, $\tlam$, respectively.
Using the same procedure for \eqref{wi_lin_4_2}--\eqref{wi_lin_4_4},
we obtain, by the orthogonality of Legendre polynomials, the expression
\begin{align}\label{aj1k-1}
\wo{v}\vertce = \sum_{m=k-1}^{k} v^1_{j,m} \lenj{m}(x),\quad v = u, p, q, r,
\end{align}
where
\begin{align}\label{aj1k-2}
u^1_{j,k-1} = \frac{\bhj \bar{p}_{j,k}}{2k+1},
p^1_{j,k-1} = \frac{\bhj \bar{q}_{j,k}}{2k+1},
q^1_{j,k-1} = \frac{\bhj \bar{r}_{j,k}}{2k+1},
r^1_{j,k-1} = -\frac{\bhj \left((\bar{u}_t)_{j,k} + \bar{p}_{j,k} +  \bar{q}_{j,k} \right) }{2k+1}.
\end{align}

In what follows, let us concentrate on  $v^1_{j,k}$, for $v = u, p, q, r$ and $j\in \z{N}$. Using the boundary collocations in \eqref{wi_lin_4} and the fact that $\lenj{m}(\xjpp) = 1$, $\lenj{m}(\xjmm) = (-1)^m$, we get
\begin{align}\label{sys i=1}
		\sigma v_{j,k}^{1} + (-1)^k \tilsig v^{1}_{j+1,k} = (-1)^k \tilsig v^{1}_{j+1,k-1} -  \sigma v_{j,k-1}^{1} \triangleq \kappa_v^j,
\end{align}
where $v_{N+1,k}^{1} = v_{1,k}^{1}$ and $ \sigma = \theta,\tlam,\lambda,\tilth $.
Consequently, the linear system \eqref{sys i=1} can be written in the matrix-vector form
\begin{align}\label{matri}
		A_{v}\vec{v}^1_{k} = \vec{\kappa}_v,
\end{align}
where
\begin{align*}
	\vec{v}^1_{k} &= \left(v^1_{1,k}, v^1_{2,k}, \ldots, v^1_{N,k}\right)^\top, \quad v = u, p, q, r,
\end{align*}
and
\begin{align*}
	A_{v} &= \text{circ}\Big(\sigma, (-1)^k\tilsig, 0, \ldots, 0\Big), \quad v = u,p,q,r, \quad \sigma = \theta,\tlam,\lambda,\tilth
\end{align*}
is an $N\times N$ circulant matrix and $\vec{\kappa}_v = \left(\kappa_v^1,\kappa_v^2,\ldots, \kappa_v^N\right)^\top.$
The determinant of  $A_{\sigma}$ is
\begin{align*}
	\mo{A_{v}} = \sigma^N(1-\mu_\sigma^N),\quad v = u,p,q,r,\quad \sigma = \theta,\tlam,\lambda,\tilth,
\end{align*}
where $ \mu_{\sigma} = (-1)^{k+1}\tilsig/\sigma $. Thus, $A_{v}$ is always invertible for $\sigma \ne\ff$ and the linear system \eqref{matri} has unique solutions. Moreover, the inverse of $A_{v}$,
\begin{align*}
A_{v}^{-1} = \frac{1}{\sigma(1-\mu_\sigma^N)}\text{circ}\left( 1,\mu_{\sigma},\mu_\sigma^2,\ldots,\mu_\sigma^{N-1} \right)
\end{align*}
is also circulant.
After a direct calculation, we have
\begin{align*}
v^1_{j,k} = \frac{1}{\sigma(1-{\mu_{\sigma}}^N)} \sum^N_{m=1} \rho_{j,m}\kappa_v^m, \quad j\in \z{N},
\end{align*}
where $ \{\rho_{j,m}\}^N_{m=1} $ are entries of the $ j $-th row of $ \text{circ}\left( 1,\mu_{\sigma},{\mu_{\sigma}}^2,\ldots,{\mu_{\sigma}}^{N-1} \right) $. By \eqref{vjk esti}, \eqref{aj1k-2}--\eqref{sys i=1}, we get
\begin{align*}
	\kappa_v^j & \le  \mo{ (-1)^k \tilsig v^{1}_{j+1,k-1} - \sigma v_{j,k-1}^{1}}
	\lesssim h^{k+2} \norm{v}_{k+2,\infty }, \nonumber \\
v^1_{j,k} & \le  \frac{\mo{\kappa_v^j}}{\sigma(1-\mu_\sigma^N)}\cdot \frac{1-\mu_\sigma^N}{1-\mu_\sigma}
	=  \frac{\mo{\kappa_v^j}}{\sigma(1-\mu_{\sigma})}
	\lesssim h^{k+2} \norm{v}_{k+2,\infty}.
\end{align*}
This, together with \eqref{aj1k-1} and \eqref{vjk esti}, produces
\begin{align*}
\normice{\pnt \wo{u}} &\lesssim h \mo{\pnt \bar{p}_{j,k}}\lesssim h^{k+2}\norm{\pnt p}_{k+1,\infty,\ce}\lesssim h^{k+2}\norm{\pnt u}_{k+2,\infty},\\
\normice{\pnt \wo{p}} &\lesssim h \mo{\pnt \bar{q}_{j,k}}\lesssim h^{k+2}\norm{\pnt q}_{k+1,\infty,\ce}\lesssim h^{k+2}\norm{\pnt p}_{k+2,\infty},\\
\normice{\pnt \wo{q}} &\lesssim h \mo{\pnt \bar{r}_{j,k}}\lesssim h^{k+2}\norm{\pnt r}_{k+1,\infty,\ce}\lesssim h^{k+2}\norm{\pnt q}_{k+2,\infty},\\
\norm{\pnt \wo{r}}_{\infty,I_j} &
\lesssim h\left(\mo{\pnt\bar{p}_{j,k}} + \mo{\pnt\bar{q}_{j,k}} + \mo{\pnt(\bar{u}_t)_{j,k}}\right)\\
	&\lesssim h^{k+2}\left(\norm{\pnt p}_{k+1,\infty,\ce} + \norm{\pnt q}_{k+1,\infty,\ce} + \norm{\pnt u_t}_{k+1,\infty,\ce}\right)\\
	&\lesssim h^{k+2}\norm{\pnt r}_{k+2,\infty},
\end{align*}
and thus \eqref{pro cf fom} holds for $i=1$.

\textbf{Step 2:} Assume that \eqref{pro cf fom} is valid for $i \le k-1$ and we want to show it still holds for $i+1$. By induction hypothesis together with an argument similar to that in deriving \eqref{aj1k-1}, we have
\begin{align*}
\left.\wi{v}\right|_{\ce} = \sum_{m=k-i}^k v^i_{j,m}, \lenj{m}(x), \quad v=u, p, q, r.
\end{align*}
Choosing $v = \lenj{m}$, $m\le k-1$, $ j \in \z{N} $ in \eqref{wi_lin_4_1} and recalling \eqref{inv operator}, we obtain
\begin{align*}
 \left(\wipo{u} - \bhj \di \wi{p},\ \lenj{m}\right)_j
=  \left(\sum_{m=0}^{k} u^{i+1}_{j,m} \lenj{m} - \bhj \sum_{m = k-i}^{k}\frac{p^i_{j,m} \left(L_{j,m+1}-L_{j,m-1}\right) }{2m+1} ,\lenj{m}\right)_j =0.
\end{align*}
Using the same procedure for \eqref{wi_lin_4_2}--\eqref{wi_lin_4_4}, we get, by the orthogonality of Legendre polynomials,
\begin{align*}
	\wipo{v}\vertce = \sum_{m=k-i-1}^{k} v^{i+1}_{j,m} \lenj{m}(x), \quad v = u,p,q,r
\end{align*}
where
\begin{align*}
v^{i+1}_{j,k-i-1} &= -\frac{\bhj \hat{v}^i_{j,k-i}}{2(k-i)+1},& &v = u,p,q,\\
v^{i+1}_{j,k-i} &= -\frac{\bhj \hat{v}^i_{j,k-i+1}}{2(k-i)+3},& &\hat{v} = p,q,r,\\
v^{i+1}_{j,m} &= - \frac{\bhj  \hat{v}^i_{j,m+1}}{2m+3} + \frac{ \bhj \hat{v}^i_{j,m-1}}{2m-1},& &k-i+1 \le m\le k-1,
\end{align*}
and
\begin{align*}
	r^{i+1}_{j,k-i-1} &=  \bhj \frac{ (u_t)^i_{j,k-i} + p^i_{j,k-i} + q^i_{j,k-i} }{2(k-i)+1},\\
	r^{i+1}_{j,k-i}  &=  \bhj \frac{  (u_t)^i_{j,k-i+1} + p^i_{j,k-i+1} + q^i_{j,k-i+1}  }{2(k-i)+3}  ,\\
	r^{i+1}_{j,m}  & = \bhj \frac{(u_t)^i_{j,m+1} + p^i_{j,m+1} + q^i_{j,m+1} }{2m+3} \\
& \ind - \bhj \frac{(u_t)^i_{j,m-1} + p^i_{j,m-1} + q^i_{j,m-1} }{2m-1},  k-i+1 \le m\le k-1.
\end{align*}	
Thus, we obtain the following system
\begin{align}\label{sysio}
	\sigma v^{i+1}_{j,k} + \tilsig  (-1)^k v^{i+1}_{j+1,k} =  -\sigma \sum_{m=k-i-1}^{k-1} v^{i+1}_{j,m}  -\tilsig \sum_{m=k-i-1}^{k-1}(-1)^{m} v^{i+1}_{j+1,m},
\end{align}
where $v_{N+1,k}^{i+1} = v_{1,k}^{i+1}$, $v = u,p,q,r$ and $ \sigma = \theta,\tlam,\lambda,\tilth $, respectively. Again, the above linear system can be rewritten as a matrix-vector form as that in \eqref{matri}. When $\sigma \ne \ff$, we can establish the uniqueness, existence of \eqref{sysio} and obtain
\begin{align}\label{klee}
	v^{i+1}_{j,k} \lesssim h^{k+i+2}\norm{v}_{k+i+2,\infty}.
\end{align}
It is easy to show, for $\pnt \wipo{r}$ with $n=0,1$, that
\begin{align*}
\norm{\pnt \wipo{r}}_{\infty,I_j}
&\lesssim \sum_{m=k-i-1}^k \mo{\pnt r^{i+1}_{j,m}} \\
&\lesssim h \left(\sum_{m=k-i}^k \mo{\pnt (u_t)^i_{j,m}} + \sum_{m=k-i}^k \mo{\pnt p^i_{j,m}} + \sum_{m=k-i}^k \mo{\pnt q^i_{j,m}} \right)\\
&\lesssim h^{k+i+2} \left( \norm{\pnt u_t}_{k+i+1,\infty}  + \norm{\pnt  p}_{k+i+1,\infty}  + \norm{\pnt q}_{k+i+1,\infty} \right)\\
&\lesssim h^{k+i+2}\norm{\pnt r}_{k+i+2,\infty} .
\end{align*}
Analogously,
\begin{align}\notag
\norm{\pnt \wipo{u}}_{\infty,I_j} & \lesssim h^{k+i+2}\norm{\pnt u}_{k+i+2,\infty},\\\notag
\norm{\pnt \wipo{p}}_{\infty,I_j} & \lesssim h^{k+i+2}\norm{\pnt p}_{k+i+2,\infty},\\\notag
\norm{\pnt \wipo{q}}_{\infty,I_j} & \lesssim h^{k+i+2}\norm{\pnt q}_{k+i+2,\infty}.
\end{align}
Therefore, \eqref{pro cf fom} is valid for $i+1$ with $v = u, p, q, r$. This finishes the proof of Lemma \ref{pro cf}.
\end{proof}

\subsection{A superconvergent bound of interpolation errors}\label{rhs}
By the LDG scheme \eqref{num_lin_4_LDG} and Galerkin orthogonality, we get error equations
\begin{subequations} \label{orthj}
\begin{align}\label{orthj1}
		(\eut ,v)_j  - \hjop{\theta}{e_u}{v} - \hjop{\tlam}{e_p}{v} - \hjop{\tilth}{e_r}{v} &= 0, \\ \label{orthj2}
		(e_r,\phi)_j + \hjop{\lambda}{e_q}{\phi} &= 0, \\ \label{orthj3}
		(e_q,\psi)_j + \hjop{\tlam}{e_p}{\psi} &= 0, \\ \label{othj4}
		(e_p,\zeta)_j + \hjop{\theta}{e_u}{\zeta} &= 0,
\end{align}
\end{subequations}
where $ e_v = v - v_h $ with $ v = u,p,q,r $. We use the following decomposition
\begin{subequations}
\begin{align}\label{error_lin_4}
	e_v  = v - \il{v} + \il{v} - v_h &\triangleq \epsilon_v + \bar{e}_v,
\end{align}
and introduce the interpolation function
\begin{equation}\label{vil}
\il{v} = \psig v - \wl{v} ~~(\ell\in \z{k})
\end{equation}
\end{subequations}
with the GGR projection $\psig v$  ($\sigma = \theta,\tlam,\lambda,\tilth$) given in \eqref{ggr_pro} and the correction function $ \wl{v}$ defined in \eqref{wl}.

A superconvergent bound of interpolation errors is presented in the following lemma.
\begin{lemma}\label{pro of interpolation functions}
Assume that $u$ is the exact solution of \eqref{eq_lin_4_}, which is sufficiently smooth, e.g., $u \in W^{k+\ell+4,\infty}(\ih)$ and $u_t \in W^{k+\ell+1,\infty}(\ih)$. For $\ell\in \z{k}$, $\il{v}$ ($v = u,p,q,r$) are the interpolation functions defined in \eqref{vil}. Then, for $\phi \in V_h$, we have
\begin{subequations}\label{u with uli}
\begin{align}
|\left(\epsilon_p, \phi\right)_j + \hjop{\theta}{\epsilon_u}{\phi}| &\lesssim h^{k+\ell +1}\norm{u}_{k+\ell+2,\infty}\norm{\phi}_{1,\ce} , \label{key1}\\
|\left(\epsilon_q, \phi\right)_j + \hjop{\tlam}{\epsilon_p}{\phi}|&\lesssim h^{k+\ell +1}\norm{u}_{k+\ell+3,\infty}\norm{\phi}_{1,\ce}, \label{key2}\\
|\left(\epsilon_r, \phi\right)_j +  \hjop{\lambda}{\epsilon_q}{\phi}| &\lesssim h^{k+\ell +1}\norm{u}_{k+\ell+4,\infty}\norm{\phi}_{1,\ce}, \label{key3}\\
|\left(\epsilonut, \phi\right)_j -  \hjop{\theta}{\epsilon_u}{\phi} -  \hjop{\tlam}{\epsilon_p}{\phi}- \hjop{\tiltheta}{\epsilon_r}{\phi}| &\lesssim h^{k+\ell +1}\norm{u_t}_{k+\ell+1,\infty}\norm{\phi}_{1,\ce}.\label{key4}
\end{align}
\end{subequations}
\end{lemma}
\begin{proof}	
In what follows, let us show \eqref{key4} only, and proofs for \eqref{key1}--\eqref{key3} are analogous. To do that, for $v = u,p,q,r$ and $\sigma=\theta,\tilde{\lambda},\lambda,\tilth $, since $\epsilon_v = v - \il{v} = v- \psig v + \wl{v}$, we deduce from \eqref{dgo} and \eqref{wi_lin_4} that
$$
\hjop{\sigma}{\epsilon_v}{\phi} = \hjop{\sigma}{v- \psig v}{\phi} + \hjop{\sigma}{\wl{v}}{\phi} = \hjop{\sigma}{\wl{v}}{\phi}
=(\wl{v}, \phi_x)_j.
$$
Consequently, for the left hand side of \eqref{key4}, one has
\begin{subequations}
\begin{align}
\mathcal S & \triangleq \left(\epsilonut, \phi\right)_j -  \hjop{\theta}{\epsilon_u}{\phi} -  \hjop{\tlam}{\epsilon_p}{\phi}- \hjop{\tiltheta}{\epsilon_r}{\phi} \nonumber  \\
& = \left(\epsilonut, \phi\right)_j - \left( \wl{u} +  \wl{p} + \wl{r}, \phi_x\right)_j \nonumber \\
& = \Big(\big( \wz{u} +\sum_{i =1}^{\ell} \wi{u}\big)_t, \phi\Big)_j - \left( \wl{u} +  \wl{p} + \wl{r}, \phi_x\right)_j. \label{sa}
\end{align}
For $i\in \z{\ell}$, we now employ integration by parts and the definition of correction functions in \eqref{wi_lin_4_4} to conclude that
\begin{equation}\label{sb}
(\partial_t\wimo{u}, \phi)_j
		= \bhj \left(\left(\di \partial_t\wimo{u}\right)_x, \phi\right)_j = -\bhj (\di \partial_t \wimo{u}, \phi_x)_j = ( \wi{u} +  \wi{p} + \wi{r}, \phi_x)_j,
\end{equation}
\end{subequations}
where we have also used the fact that
\begin{align*}
\bhj \di  \partial_t\wimo{u}(\xjp^-) & = ( \partial_t\wimo{u},1)_j = 0, \\
 \bhj \di \partial_t \wimo{u}(\xjm^+) & = 0
\end{align*}
implied by the definition of integral operator in \eqref{dxi} and the orthogonality of $\partial_t \wimo{u}$ in \eqref{pro of correction function 2}.

If we now substitute \eqref{sb} with $i\in \z{\ell}$ into \eqref{sa}, we obtain
\begin{equation}\label{s}
\mathcal S = ( \partial_t \wll u, \phi)_j,
\end{equation}
which, in combination with \eqref{pro of correction function 1}, gives us
$$
|\mathcal S | \lesssim h^{k+\ell +1}\norm{u_t}_{k+\ell+1,\infty}\norm{\phi}_{1,\ce}.
$$
This completes the proof of \eqref{key4} and thus Lemma \ref{pro of interpolation functions}.
\end{proof}

\section{Superconvergence}\label{sc analysis}
In this section, we first introduce a suitable numerical initial condition satisfying superconvergent property in Section \ref{ini}, then show supercloseness between interpolation functions and LDG solutions in Section  \ref{sc}, and derive superconvergence concerning numerical flux, cell averages and generalized Radau points in Section \ref{superconvergence}.

By the error decomposition \eqref{error_lin_4} and using the same argument as that in deriving \eqref{s} in the proof of Lemma \ref{pro of interpolation functions}, we sum the error equations \eqref{orthj} over all $j$ to obtain
\begin{subequations} \label{orthde}
	\begin{align}\label{orthde1}
		(\bet{u},v)   - \hop{\theta}{\be{u}}{v}
		- \hop{\tlam}{\be{p}}{v} - \hop{\tiltheta}{\be{r}}{v} + (\partial_t \wll u,v) &=0, \\ \label{orthde2}
		(\be{r},\phi)  + \hop{\lambda}{\be{q}}{\phi}+ (\wll{r},\phi) &=0, \\ \label{orthde3}
		(\be{q},\psi)  + \hop{\tlam}{\be{p}}{\psi}+ (\wll{q},\psi) &=0, \\ \label{orthde4}
		(\be{p},\zeta)  + \hop{\theta}{\be{u}}{\zeta}+ (\wll{p},\zeta) &=0.
	\end{align}
\end{subequations}

\subsection{The numerical initial condition}\label{ini}
To be compatible with superconvergent property, a suitable choice of numerical initial condition is constructed as follows. For $u_0 \in W^{k+\ell+4, \infty}(\ih)$ and $u_t(0) \in W^{k+\ell+1, \infty}(\ih)$ ($\ell\in \z{k}$), choose
\begin{align}\label{ini con}
		r_h(x,0) = \il{r}(x,0) = \ptt r_0(x) - \wl{r}(x,0),\quad r_0 = \partial_x^3 u_0(x)
\end{align}
with $ u_h(x,0)$, $p_h(x,0)$ and $q_h(x,0)$ being the solutions to
\begin{subequations}\label{im}
	\begin{align}
		(r_h,\phi)_j + \hjop{\lambda}{q_h}{\phi} &= 0,\label{im1}\\
		(q_h,\psi)_j + \hjop{\tlam}{p_h}{\psi} &= 0,\label{im2}\\
		(p_h,\zeta)_j + \hjop{\theta}{u_h}{\zeta} &= 0,\label{im3}
	\end{align}
\end{subequations}
where $q_0 = \partial_x^2 u_0(x), p_0 = \partial_x u_0(x)$.
Existence, uniqueness as well as superconvergent initial error estimates for the above numerical initial condition are established in the following lemma.
\begin{lemma}\label{lemma sc}
Suppose that the initial condition $u_0$ is periodic satisfying $u_0 \in W^{k+\ell+4, \infty}(\ih)$ and $u_t(0) \in W^{k+\ell+1, \infty}(\ih)$. Assuming that interpolation functions $\il{v}$ and errors $\be{v}$ ($v = u,p,q,r$) are defined in \eqref{error_lin_4}--\eqref{vil} $(\ell \in \z{k})$, then the numerical initial conditions in \eqref{ini con}--\eqref{im} are uniquely determined and satisfy
\begin{align*}
	\norm{\be{u}(0)} + \norm{\bet{u}(0)} + \norm{\be{p}(0)} + \norm{\be{q}(0)} + \norm{\be{r}(0)} \le Ch^{k+\ell+1},
\end{align*}
where $C$ depends on $\norm{u}_{k+\ell+4,\infty}$ and $\norm{u_t}_{k+\ell+1,\infty}$, but is independent of $h$.
\end{lemma}
\begin{proof}
Let us start by showing unique existence, and taking $q_h(x,0)$ as an example.
To do that, we need the following conservation property of $q_h(x,0)$,
	\begin{align}\label{in1}
		\intid{(q_0 - q_h)} = 0,
	\end{align}
which is obtained by taking $\psi = 1$ in \eqref{im2}, summing over all $j$ and using the definition of DG operator in \eqref{dgo}, i.e.,
	\begin{align*}
		\intid{q_h} + \sumjn \big(-p_h^{\tla}{\ljp} + p_h^{\tla}{\ljm}\big) = 0,
	\end{align*}
in combination with periodic boundary conditions and Galerkin orthogonality.
For $t=0$, suppose $q_h^1$ and $ q_h^2$ are the solutions of \eqref{im1} with $r_h$ satisfying \eqref{ini con}.
Denoting $ w_h = q_h^1 - q_h^2 \in V_h$, it follows from \eqref{im1} and \eqref{in1} that
\begin{align*}
	\hjop{\lambda}{w_h}{\phi} = 0,\quad \intid{w_h} =0,
\end{align*}
which, by letting $\phi = w_h$, summing over all $j$ and using the identity \eqref{DG operator}, implies
$$
\jumpmos{w_h} = 0.
$$
This indicates that $w_h$ is constant in $I$. Since $\intid{w_h} =0$, we conclude that
$$
	w_h \equiv 0.
$$
Therefore, $q_h(x,0)$ is unique, and thus for $p_h(x,0)$. Since \eqref{im} is a linear system, the existence follows immediately.
When $t=0$, the scheme \eqref{LDG1} is still valid due to the continuity of numerical solutions with respect to time. This allows us to derive the conservation property of $u_h(x,0)$, and thus unique existence follows.

We now move on to the estimate of $ \norm{\bez{q}} $, and the estimates to $ \norm{\bez{p}} $, $ \norm{\bez{u}} $ are analogous. By using \eqref{orthde2} and  $ \be{r}(0) = 0 $ in \eqref{ini con}, we have
\begin{align*}
	\hop{\lambda}{\bez{q}}{\phi} = -(\wllz{r},\phi),
\end{align*}
which, by Lemma \ref{derivative jump}, yields
\begin{align}
	\norm{\bexz{q}} + h^{-\ff}\jumpmo{\bez{q}} \lesssim \norm{\wllz{r}}.\label{r co2}
\end{align}
Using the orthogonality of $\pl q$ in \eqref{ggr_pro_1} and  $\wi{q}$ in \eqref{pro of correction function 2}, we have
\begin{align*}
	\intijd{\be{q}} = \intijd{\leri{ \pl q - \wl{q} - q_h} } = \intijd{ \leri{q - \wll{q} - q_h} }.
\end{align*}
Summing the above equation over all $j$ and taking into account \eqref{in1}, we arrive at
\begin{align}
	\intid{\bez{q}} = \intid{\leri{ q_0 - \wllz{q} - q_h(0)}} = -\intid{\wllz{q}}.\label{r co3}
\end{align}
We are now ready to estimate $ \norm{\bez{q}} $. It reads
\begin{align*}
	\norm{\bez{q}} &\le \norm{\bez{q} - \frac{1}{|I|}\intid{\bez{q}} } + |\frac{1}{|I|}\intid{\bez{q}}|\\
	&\lesssim \norm{\bexz{q}} + h^{-\ff}\jumpmo{\bez{q}}
	+ |\intid{\wllz{q}}|\\
	&\lesssim \norm{\wllz{r}} +  \norm{\wll{q}(0)}_{\infty},
\end{align*}
where in the second step we have used \eqref{r co3} and the discrete Poincar\'{e}  inequality with $\zeta = \bez{q} - \frac{1}{|I|}\intid{\bez{q}} $ in \eqref{dp},
and in the last step we have employed \eqref{r co2}. Consequently, by  \eqref{pro of correction function 1} in Lemma \ref{pro cf}, we get
\begin{align*}
	\norm{\bez{q}} \le  Ch^{k+\ell+1},
\end{align*}
where $C$ depends on $\norm{u}_{k+\ell+4,\infty}$, but is independent of $h$.

To finish the proof of Lemma \ref{lemma sc}, it remains to show a bound for $\norm{\bet{u}(0)}$. Due to the continuity  with respect to time, \eqref{orthde1} is still valid for $t=0$.  Since $\bar{e}_r(0) = 0$,  we rewrite \eqref{orthde1} to get
\begin{align*}
	(\bet{u},v)  &= -(\wllt{u},v) + \hop{\theta}{\be{u}}{v}
	+ \hop{\tlam}{\be{p}}{v}\\
	&= -(\wllt{u},v)  - (\be{p}, v) - (\wll{p},v)  - (\be{q},v) - (\wll{q},v),
\end{align*}
where we have also used \eqref{orthde3}--\eqref{orthde4}.
Letting $ v=\betz{u} $, by using Young's inequality and \eqref{pro of correction function 1} in Lemma \ref{pro cf}, we have, at $ t=0 $,
\begin{align*}
	\norms{\bet{u}}  &\le \frac{1}{16} \norms{\bet{u}} +  \frac{1}{16} \norms{\bet{u}} + \norms{\be{p}} +  \frac{1}{4}\norms{\bet{u}} +  \frac{1}{8}\norms{\bet{u}} \\
	&\quad \!~+ \norms{\be{q}} +  \frac{1}{4}\norms{\bet{u}} + \chtkl\\
	&\le \frac{3}{4} \norms{\bet{u}} + \norms{\be{p}} + \norms{\be{q}}  + \chtkl,
\end{align*}
which is,
\begin{align*}
	\norms{\bet{u}} \le  4\norms{\be{p}} + 4\norms{\be{q}}  + \chtkl.
\end{align*}
Then, by using the estimates of $ \norm{\bez{q}} $ and $ \norm{\bez{p}} $, we have
\begin{align*}
	\norm{\bet{u}(0)} \le Ch^{k+\ell+1},
\end{align*}
where $C$ depends on $\norm{u}_{k+\ell+4,\infty}$ and $\norm{u_t}_{k+\ell+1,\infty}$, but is independent of $h$.
This completes the proof of Lemma \ref{lemma sc}.
\end{proof}

\subsection{Supercloseness}\label{sc}
The supercloseness between interpolation functions and LDG solutions is given in the following theorem.
\begin{theorem}\label{scth}
Suppose $u$ is the exact solution of the fourth-order problem \eqref{eq_lin_4_} with periodic boundary conditions, which is sufficiently smooth, e.g., $u \in W^{k+\ell+4,\infty}(\ih)$, $u_t \in W^{k+\ell+3,\infty}(\ih)$. Assume that $v_h$ ($v = u,p,q,r$) are LDG solutions to \eqref{num_lin_4_LDG} with generalized fluxes \eqref{gflux_lin_4_LDG} and $ \theta=\lambda>\ff$. Let  $\be{v}$ and interpolation functions $\il{v}$  be defined in \eqref{error_lin_4}--\eqref{vil} $(\ell\in \z{k})$. Then, under the numerical initial condition \eqref{ini con}--\eqref{im}, we have the following supercloseness result
\begin{align}\label{supercloseness}
	\norm{\be{u}(t)} + \norm{\be{p}(t)} + \leri{\inttdtau{ (\norms{\be{q}} + \norms{\be{r}}) }}^{\ff} \le \chkl,
\end{align}
where $C$ depends on $\norm{u}_{k+\ell+4,\infty}$ and $\norm{u_t}_{k+\ell+3,\infty}$, but is independent of $h$.
\end{theorem}

\begin{proof}
First, taking $(v,\phi,\psi,\zeta) = (\be{u},\be{p},\be{q},-\be{r})$ in \eqref{orthde1}--\eqref{orthde4} and adding them together, by using the generalized skew-symmetry property in \eqref{DG operator}, namely,
\begin{align}
	\hop{\theta}{\be{u}}{\be{r}} + \hop{\tilth}{\be{r}}{\be{u}} = 0,~~
	\hop{\lambda}{\be{q}}{\be{p}} + \hop{\tlam}{\be{p}}{\be{q}} = 0,\label{ssqp}
\end{align}
 we get
\begin{align}\label{uq1}
	\ff\fd \norms{\be{u}} \!+\! \norms{\be{q}}  =-(\wllt{u},\be{u}) -  (\wll{r},\be{p}) - (\wll{q},\be{q}) + (\wll{p},\be{r}) \!+\! \hop{\theta}{\be{u}}{\be{u}} \!+\! \hop{\tlam}{\be{p}}{\be{u}}.
\end{align}
Utilizing identity \eqref{DG operator} with the same weight $\theta$, we have
\begin{align*}
	\hop{\theta}{\be{u}}{\be{u}} = \big(\ff - \theta\big) \sumjn \jumpsjp{\be{u}} \le 0,
\end{align*}
since $\theta>\ff$.
Taking $\psi = \be{u}$ in \eqref{orthde3}, we obtain
\begin{align*}
	 \hop{\tlam}{\be{p}}{\be{u}} = - (\be{q},\be{u}) - (\wll{q},\be{u}).
\end{align*}
Consequently, \eqref{uq1} becomes
\begin{align*}
	\ff\fd \norms{\be{u}} + \norms{\be{q}} &\le -(\wllt{u},\be{u}) -  (\wll{r},\be{p}) - (\wll{q},\be{q}) + (\wll{p},\be{r}) - (\be{q}, \be{u}) - (\wll{q},\be{u}).
\end{align*}
By using Young's inequality and \eqref{pro of correction function 1} in  Lemma \ref{pro cf}, we have
\begin{align*}
	\ff\fd \norms{\be{u}} \!+\! \norms{\be{q}}
	& \!\le \! \ff\norms{\be{u}} + \norms{\be{p}} + \frac{1}{8}\norms{\be{q}} +\frac{1}{4} \norms{\be{r}}  \!+ \! \frac{1}{8}\norms{\be{q}} + 2\norms{\be{u}} + \ff\norms{\be{u}} + \chtkl \!,
\end{align*}
which is,
\begin{align}\label{teu}
	\ff\fd \norms{\be{u}} + \norms{\be{q}}
	&\le 3\norms{\be{u}} + \norms{\be{p}} + \frac{1}{4}\norms{\be{q}} +\frac{1}{4} \norms{\be{r}} + \chtkl.
\end{align}

Next, we take the time derivative of \eqref{orthde4} and choose $(v,\phi,\psi,\zeta) = (-\be{q},\be{r},\bet{u},\be{p})$ in \eqref{orthde1}--\eqref{orthde3} and the newly obtained \eqref{orthde4}.
Summing them together and using the generalized skew-symmetry property in \eqref{DG operator} with $\theta = \lambda$, namely,
\begin{align*}
	\hop{\theta}{\bet{u}}{\be{p}} + \hop{\tlam}{\be{p}}{\bet{u}} = 0,~~
	\hop{\lambda}{\be{q}}{\be{r}} + \hop{\tilth}{\be{r}}{\be{q}} = 0,
\end{align*}
we obtain
\begin{align}\label{pr1}
	\ff\fd \norms{\be{p}} \!+\! \norms{\be{r}} \!=\!(\wllt{u},\be{q}) \!-\!  (\wll{r},\be{r}) \!-\! (\wll{q},\bet{u}) \!-\! (\wllt{p},\be{p}) \!-\! \hop{\theta}{\be{u}}{\be{q}} \!-\! \hop{\tlam}{\be{p}}{\be{q}}.
\end{align}
Taking $\zeta = \be{q}$ in \eqref{orthde4}, we get
\begin{align*}
	-\hop{\theta}{\be{u}}{\be{q}} = (\be{p},\be{q}) + (\wll{p},\be{q}).
\end{align*}
Taking $\phi = \be{p}$ in \eqref{orthde2}, and using the  generalized skew-symmetry property in \eqref{ssqp}, we derive
\begin{align*}
 - \hop{\tlam}{\be{p}}{\be{q}} = \hop{\lambda}{\be{q}}{\be{p}} =  - (\be{r},\be{p}) - (\wll{r},\be{p}).
\end{align*}
Consequently, \eqref{pr1} becomes
\begin{align*}
	\ff\fd \norms{\be{p}} + \norms{\be{r}} &\le (\wllt{u},\be{q}) -  (\wll{r},\be{r}) - (\wll{q},\bet{u}) - (\wllt{p},\be{p}) \\
	&\quad \!~ + (\wll{p},\be{q}) + (\be{p}, \be{q})
	- (\wll{r},\be{p}) - (\be{r}, \be{p}).
\end{align*}
Utilizing Young's inequality and \eqref{pro of correction function 1} in  Lemma \ref{pro cf}, we obtain
\begin{align}
	\ff\fd \norms{\be{p}} + \norms{\be{r}}
	&\le  \frac{1}{8}\norms{\be{q}} + \frac{1}{4}\norms{\be{r}} + \Psi + \frac{1}{2}\norms{\be{p}} + \chtkl \nonumber \\
	&\ind + \frac{1}{8}\norms{\be{q}} + \frac{1}{4}\norms{\be{q}} + \norms{\be{p}} +  \frac{1}{2}\norms{\be{p}} +  \frac{1}{4}\norms{\be{r}} + \norms{\be{p}}, \nonumber\\
	&\le \Psi + 3\norms{\be{p}} +
	 \frac{1}{2}\norms{\be{q}} + \frac{1}{2}\norms{\be{r}}  + \chtkl,\label{teut}
\end{align}
where
$$
\Psi = - (\wll{q},\bet{u}),
$$
and $C$ depends on $\norm{u}_{k+\ell+4,\infty}$ and $\norm{u_t}_{k+\ell+1,\infty}$, but is independent of $h$.

Now, summing \eqref{teu} and \eqref{teut} together,  we have
\begin{align}\label{upqr1}
 \ff\fd \leri{\norms{\be{u}} + \norms{\be{p}}} + \frac{1}{4}\leri{\norms{\be{q}} + \norms{\be{r}}}
\le \Psi + 3\norms{\be{u}} + 4\norms{\be{p}} + \chtkl,
\end{align}
where $\Psi$ satisfies, by integration by parts with respect to time,
\begin{align*}
\inttdtau{\Psi}	= -\inttdtau{ \intid{\wll{q}\bet{u}} }
= \inttdtau{ \intid{\wllt{q}\be{u}} } - \intid{\wll{q}(t)\be{u}(t)} + \intid{\wll{q}(0)\be{u}(0)}.
\end{align*}
It implies that, by using Young's inequality, \eqref{pro of correction function 1} in Lemma \ref{pro cf} and the estimate of $\norm{\be{u}(0)}$ in Lemma \ref{lemma sc},
\begin{align}\label{psi}
	\inttdtau{\Psi}	
	\le \inttdtau{ \norms{\be{u}} } + \frac{1}{4}\norms{\be{u}(t)}  + \chtkl,
\end{align}
where $C$ depends on $\norm{u}_{k+\ell+3,\infty}$ and $\norm{u_t}_{k+\ell+3,\infty}$, but is independent of $h$.
Integrating \eqref{upqr1} with respect to time from $ 0 $ to $ t $, we get, by \eqref{psi}
\begin{align*}
	&\frac{1}{4}\norms{\be{u}(t)} + \ff\norms{\be{p}(t)} + \frac{1}{4}\inttdtau{\leri{\norms{\be{q}(t)} + \norms{\be{r}(t)}}} \\
	&  \le  4\inttdtau{ \leri{\norms{\be{u}} + \norms{\be{p}}}}
	 +  \ff \leri{\norms{\be{u}(0)} + \norms{\be{p}(0)}} + \chtkl.
\end{align*}
This, together with estimates of $\norm{\be{u}(0)}$ and $\norm{\be{p}(0)}$ in Lemma \ref{lemma sc}, produces
\begin{align*}
	\frac{1}{4}\norms{\be{u}(t)} + \ff\norms{\be{p}(t)}  + \frac{1}{4}\inttdtau{\leri{\norms{\be{q}(t)} + \norms{\be{r}(t)}}}
	\le  4\inttdtau{ \leri{\norms{\be{u}} + \norms{\be{p}}}} + \chtkl,
\end{align*}
where  $C$ depends on $\norm{u}_{k+\ell+4,\infty}$ and $\norm{u_t}_{k+\ell+3,\infty}$, but is independent of $h$.
A straightforward application of Gronwall's inequality gives us the desired result
\eqref{supercloseness}. This completes the proof of Theorem \ref{scth}.
\end{proof}

\subsection{Superconvergence}\label{superconvergence}
To show superconvergent results at generalized Radau points, we begin by recalling the generalized Radau polynomials \cite{cao2017upwind-biased}
\begin{equation*}
	R^{\sigma}_{k+1} =
	\left\{
	\begin{array}{lr}
		L_{k+1} - (2\sigma - 1)L_k,\quad \text{when $k$ is even},  \\
		(2\sigma-1)L_{k+1}-L_k,\quad \text{when $k$ is odd}
	\end{array}
	\right.
\end{equation*}
defined in $ [-1,1] $. Then, we rescale $R^{\sigma}_{k+1} $to $ \ce $ to get $ R^{\sigma}_{j,k+1}$ $(j\in\z{N})$, and denote the roots of $R^{\sigma}_{j,k+1}$ and $ \partial_x R^{\sigma}_{j,k+1}$  by $\mathcal{R}_{j,m}^{\sigma}$ and $ \mathcal{R}_{j,m}^{\sigma,\star}$, respectively,  where $m \in \z{k}$ and $\sigma\neq\ff$. Accordingly,
for any positive weight $\theta \ne \ff$, a local projection $P_h u \in V_h$ can be defined as that in \cite{cao2017upwind-biased}, namely
\begin{align*}
&~& &~& &~&	\int_{\ce} (P_h u -u) v &= 0,&   &\forall v\in P^{k-1}(\ce),  &~& &~& &~&\\
&~& &~& &~&	\theta P_h u(\xjp^-) + \tilth P_h u(\xjm^+) &= \theta u(\xjp^-) +\tilth u(\xjm^+),&  &j\in\z{N},  &~& &~& &~&
\end{align*}
which satisfies the following lemma.
\begin{lemma}\textsuperscript{\cite{cao2017upwind-biased}}\label{phu pro}
Suppose $u\in W^{k+2,\infty}(\ih)$. For $P_h u$ defined above with $\theta \ne \ff$, we have
\begin{align*}
		\mo{(u - P_h u)(\mathcal{R}_{j,m}^{\theta})} &\lesssim h^{k+2}\norm{u}_{k+2,\infty},\\
		\mo{\partial_x(u - P_h u)(\mathcal{R}_{j,m}^{\theta,\star})} &\lesssim h^{k+1}\norm{u}_{k+2,\infty},\\
		\norm{P_h u - \pt u}_{\infty} &\lesssim h^{k+2}\norm{u}_{k+2,\infty}.
\end{align*}
\end{lemma}
In what follows, superconvergence results of the numerical flux, cell averages, generalized Radau points as well as supercloseness are presented, in which variables $u$ and $p$ are mainly considered, and the case for variables $q$ and $r$ can be established in a similar manner, essentially following Theorem \ref{scth} and \cite{liu2021sc}.
\begin{theorem}\label{scorderth}
Suppose $u$ is the exact solution of \eqref{eq_lin_4_} with periodic boundary conditions, which is sufficiently smooth, e.g., $u \in W^{2k+4,\infty}(\ih)$ and $u_t \in W^{2k+3,\infty}(\ih)$. Assume that $v_h\in V_h$ ($v=u,p,q,r$) are solutions to the LDG scheme \eqref{num_lin_4_LDG} with generalized numerical fluxes \eqref{gflux_lin_4_LDG} satisfying $ \theta=\lambda > \ff $.
Then, under the initial condition  \eqref{ini con}--\eqref{im} with  $\ell = k$, we have, for $v=u , p$

  (1) Superconvergence of the numerical flux
\begin{align*}
	\norm{e_{vn}} = \Big(\frac{1}{N} \sum_{j =1}^{N}\mo{\left(v -\hat{v}_h\right)(\xjp ,t)}^2\Big)^{\ff} \le	Ch^{2k+1},
\end{align*}
where $C$ depends on $\norm{u}_{2k+4,\infty}$ and $\norm{u_t}_{2k+3,\infty}$, but is independent of $h$.

(2) Superconvergence of the cell averages
\begin{align*}
	\norm{e_v}_c  = \Big(\frac{1}{N} \sum_{j =1}^{N}\Big|{\frac{1}{h_j}\intijd{(v-v_h)(x,t)}}\Big|^2\Big)^{\ff} \le Ch^{2k+1},
\end{align*}
where depends on $\norm{u}_{2k+4,\infty}$ and $\norm{u_t}_{2k+3,\infty}$, but is independent of $h$.

(3) When $\ell \ge 2$, the function value approximations of $v_h~(v=u,p)$ are $(k+2)$th order superconvergent at generalized Radau points $\mathcal{R}_{j,m}^{\sigma}~(\sigma = \theta, \tlam)$, and the derivative value approximations are $(k+1)$th order superconvergent at generalized Radau points $\mathcal{R}_{j,m}^{\sigma,\star}$, i.e.,
\begin{align*}
	\norm{e_{vr}} =\max_{j\in \z{N}}\mo{(v - v_h)(\mathcal{R}_{j,m}^{\sigma})} \le	Ch^{k+2} ,\quad
	\norm{e^{\star}_{vr}}= \max_{j\in \z{N}}\mo{\partial_x(v - v_h)(\mathcal{R}_{j,m}^{\sigma,\star})} \le	Ch^{k+1},
\end{align*}
where $C$ depends on $\norm{u}_{2k+4,\infty}$ and $\norm{u_t}_{2k+3,\infty}$, but is independent of $h$.

(4) Supercloseness between the GGR projection $ \psig v~(\sigma = \theta, \tlam)$ and numerical solution $v_h$ ($ v = u,p $)
\begin{equation*}
\norm{\psig v - v_h} \le Ch^{k+2},
\end{equation*}
where $C$ depends on $\norm{u}_{k+5,\infty}$ and $\norm{u_t}_{k+4,\infty}$, but is independent of $h$.
\end{theorem}

\begin{proof}
(1) Due to the boundary collocation of $\psig v$ ($ v= u,p $) in \eqref{ggr_pro_2} with $\sigma = \theta, \tilde{\lambda}$ and $\wi{v}$ ($ v = u,p $) in \eqref{wi_lin_4_1}--\eqref{wi_lin_4_2} for $ i\in \z{k}$, we get, after using the inverse inequality and Theorem \ref{scth} with $ \ell = k $
\begin{align*}
	\norm{e_{vn}} &
	= \Big(\frac{1}{N} \sum_{j =1}^{N}\mo{\left(\hat{v}^k_I -\hat{v}_h\right)(\xjp ,t)}^2\Big)^{\ff}
	\le \Big(\frac{1}{N} \sum_{j =1}^{N}C{h}^{-1}\norm{v_I^k -v_h}^2_{\ce\cup I_{j+1}}\Big)^{\ff}\\
	&\le C\norm{v^k_I - v_h}
	\le  Ch^{2k+1},\quad v= u,p,
\end{align*}
where $C$ depends on $\norm{u}_{2k+4,\infty}$ and $\norm{u_t}_{2k+3,\infty}$, but is independent of $h$.

(2) Utilizing the orthogonality of $\psig v$ ($ v= u,p $) in \eqref{ggr_pro_1} with $ \sigma = \theta, \tilde{\lambda} $ and the definition of $\il{v}$ ($ v= u,p $) in \eqref{vil} with $\ell = k$, we have
\begin{align*}
	\left(e_v,1\right)_j =   \leri{W^k_v ,1}_j + \leri{v^k_I - v_h ,1}_j, \quad \forall j \in \z{N}.
\end{align*}
It follows from \eqref{pro of correction function 1}--\eqref{pro of correction function 2} in Lemma \ref{pro cf} and Theorem \ref{scth} that
\begin{align*}
	\norm{e_v}_c \le \norm{w_v^k} + \norm{v^k_I - v_h}  \le Ch^{2k+1},\quad v= u,p,
\end{align*}
where $C$ depends on $\norm{u}_{2k+4,\infty}$ and $\norm{u_t}_{2k+3,\infty}$, but is independent of $h$.

(3) Choosing $\ell\ge2$ for  $\il{v}$ ($ v= u,p $) in Theorem \ref{scth} and using the inverse inequality, we obtain
\begin{align*}
	\norm{ \be{v} }_{\infty} \le Ch^{k+\frac{5}{2}}, ~~ v=u,p.
\end{align*}
This, together with Lemma \ref{phu pro}, produces
\begin{align*}
	\mo{(u - u_h)(\mathcal{R}_{j,m}^{\theta})}  \le \mo{(u - P_h u)(\mathcal{R}_{j,m}^{\theta})} + \norm{P_h u - P_{\theta}u}_{\infty} + \norm{\wl{u}}_{\infty} + \norm{\be{u} }_{\infty}  \le Ch^{k+2},
\end{align*}
where $C$ depends on $\norm{u}_{k+\ell+4,\infty}$ and $\norm{u_t}_{k+\ell+3,\infty}$, but is independent of $h$.
The proofs for $|(p - p_h)(\mathcal{R}_{j,m}^{\tlam})| $, $ |\partial_x(u - u_h)(\mathcal{R}_{j,m}^{\theta,\star})|$ and $ | \partial_x(p - p_h)(\mathcal{R}_{j,m}^{\tlam,\star})|$ are analogous, and details are omitted.

(4) For $\il{v}$ ($v=u,p$) with $\ell =1$, by \eqref{vil}, Theorem \ref{scth} in combination with \eqref{pro of correction function 1},  we get
\begin{align*}
\norm{\psig v - v_h} \le \norm{v^1_I - v_h} + \norm{W_v^1} \le Ch^{k+2},\quad v=u,p,
\end{align*}
where $C$ depends on $\norm{u}_{k+5,\infty}$ and $\norm{u_t}_{k+4,\infty}$, but is independent of $h$. This finishes the proof of Theorem \ref{scorderth}.
\end{proof}

\section{Extension to mixed boundary conditions}\label{extensions}
Consider the problem \eqref{eq_lin_4_} with mixed boundary conditions
\begin{align}\label{bd mixed}
	u(0,t) = g_1(t),\quad u_x(2\pi,t) = g_2(t), \quad u_{xx}(0,t) = g_3(t), \quad u_{xxx}(2\pi,t) = g_4(t).
\end{align}
The numerical fluxes are taken as
\begin{align}\label{fluxm}
	\leri{ \hat{u}_h, \hat{p}_h, \hat{q}_h, \hat{r}_h  }_{\jp}= \left\{
	\begin{aligned}
		&\big( g_1,                 p_h^+,    g_3,              r^+_h                 \big)_{\ff},  ~&   &j=0, \\
		&\big( u^{\thet}_h,p^{\tla}_h,q^{\la}_h,r^{\tilde{\thet}}_h   \big)_{\jp},  ~&   &j=1,\red{\ldots},N-1, \\
		&\big( u_h^-,    g_2,               q^-_h,  g_4                           \big)_{N+\ff},  ~&  &j=N.
	\end{aligned}
	\right.
\end{align}
Accordingly, the global projections $P_{\sigma}v$ in \eqref{ggr_pro} are modified to the piecewise global projections $\tilde P_{\sigma}v$. To be more specific,
\begin{equation}\label{ggrmixed}
\begin{aligned}
	&	\left\{
\begin{aligned}
		& (\tilde{P}_{\theta} u, \phi)_j = (u, \phi)_j, & &\forall \phi \in P^{k-1}(\ce),		\\
		& (\tilde{P}_{\theta} u )^{(\theta)}_{\jp}= u^{(\theta)}_{\jp},& &j\in \z{N-1},		 \\
		&(\tilde{P}_{\theta} u )^{-}_{N+\ff} = u^{-}_{N+\ff},& &j=N,		
\end{aligned}
	\right.
	& &
	\left\{
\begin{aligned}
		& (\tilde{P}_{\tilde{\lambda}} p, \phi)_j = (p, \phi)_j, & &\forall \phi \in P^{k-1}(\ce),	\\
		& (\tilde{P}_{\tilde{\lambda}} p )^{(\tilde{\lambda})}_{\jm} = p^{(\tilde{\lambda})}_{\jm},& &j\in \z{N}\backslash \{1\},	 \\
		& (\tilde{P}_{\tilde{\lambda}} p )^{+}_{\ff} = p^{+}_{\ff},& &j=1,
\end{aligned}
	\right.  \\
	&\left\{
\begin{aligned}
		& (\tilde{P}_{\lambda} q, \phi)_j = (q, \phi)_j, & &\forall \phi \in P^{k-1}(\ce),		\\
		& (\tilde{P}_{\lambda} q )^{(\lambda)}_{\jp}= q^{(\lambda)}_{\jp},& &j\in \z{N-1},		 \\
		&(\tilde{P}_{\lambda} q )^{-}_{N+\ff} = q^{-}_{N+\ff},& &j=N,		
\end{aligned}
	\right.
	& &	
	\left\{
\begin{aligned}
		& (\tilde{P}_{\tilde{\theta}} r, \phi)_j = (r, \phi)_j, & &\forall \phi \in P^{k-1}(\ce),	\\
		& (\tilde{P}_{\tilde{\theta}} r )^{(\tilde{\theta})}_{\jm} = p^{(\tilde{\theta})}_{\jm},& &j\in \z{N}\backslash \{1\},	 \\
		& (\tilde{P}_{\tilde{\theta}} r )^{+}_{\ff} = r^{+}_{\ff},& &j=1,
	\end{aligned}
	\right.
\end{aligned}
\end{equation}
and, following \cite{meng2016optimal,liu2021sc}, one has the following optimal approximation property for $\tilde P_{\sigma}v$ $(v= u, q, p, r; \sigma = \theta, \tilde \lambda, \lambda, \tilde \theta)$
\begin{align*}
	\norm{v-\tilde P_{\sigma}v}_{\ce} + h^{\ff}\norm{v-\tilde P_{\sigma}v}_{\infty,\ce} \le Ch^{k+\frac{3}{2}}\norm{v}_{k+1,\infty}.
\end{align*}

The functions $\wi{u}$, $\wi{p}$, $\wi{q}$ and $\wi{r} \in V_h$, $i\in \z{k}$ given below differ from definitions \eqref{wi_lin_4_1}--\eqref{wi_lin_4_4} in Section \ref{correction function} mainly in terms of boundary collocations,
\begin{subequations}\label{wim}
\begin{align}
		&(\wi{u}-\bhj \di \wimo{p},\phi)_j = 0,& &\left(\wi{u}\right)^{\thet}_{\jp} = 0, ~j\in \z{N-1}, && \left(\wi{u}\right)^{-}_{N+\ff}   = 0,\\
		&(\wi{p}-\bhj \di \wimo{q},\phi)_j = 0,& &\left(\wi{p}\right)^{\tla}_{\jp} = 0,  ~j\in \z{N-1}, && \left(\wi{p}\right)^{+}_{\ff} = 0,\\
		&(\wi{q}-\bhj \di \wimo{r},\phi)_j = 0,& &\left(\wi{q}\right)^{\la}_{\jp} = 0, ~ j\in \z{N-1}, && \left(\wi{q}\right)^{-}_{N+\ff} = 0,\\
		&( \wi{u} +  \wi{p} + \wi{r}+\bhj \di \partial_t \wimo{u},\phi)_j = 0,& & \left(\wi{r}\right)^{\tthet}_{\jp}  =0, ~j\in \z{N-1}, && \left(\wi{r}\right)^{+}_{\ff}  =0,
	\end{align}
\end{subequations}
where $ \phi \in P^{k-1}(\ce)$, $j\in \z{N}$ and
\begin{align*}
	\wz{u} = u-\tpt u,\quad \wz{p} = p-\tptl p,\quad \wz{q} = q-\tpl q,\quad \wz{r} = r-\tptt r.
\end{align*}
For $\ell\in \z{k}$, we define the correction functions as
\begin{align}\label{wlmixed}
	\wl{v}  = \sum_{i =1}^{\ell} \wi{v}, ~~ v = u,p,q,r,
\end{align}
and the interpolations functions as
\begin{align}\label{ilm}
	\il{v} = \tilde{P}_{\sigma} v - \wl{v}, ~~	v = u,p,q,r,
\end{align}
with $\tilde{P}_{\sigma} v$ ($  \sigma = \theta,\tlam,\lambda,\tilth$) being projections defined in \eqref{ggrmixed}.

Using the procedure similar to that in Section \ref{correction function}, we can obtain the existence, uniqueness, superconvergence property and orthogonality for functions $\wi{v}$ ($v = u,p,q,r$) defined in \eqref{wim} with $ i\in \z{k} $, essentially following Lemma \ref{pro cf}. The main difference is that systems \eqref{sys i=1} and \eqref{sysio} can be decoupled since the circulant matrices now reduce to
	{\small
	\begin{equation*}
	 A_{u}=\left(\begin{array}{ccccc}
			\theta & (-1)^k\tilde{\theta} & & & \\
			& \theta & (-1)^k\tilde{\theta} & & \\
			& & \ddots & \ddots & \\
			& & & \theta & (-1)^k\tilde{\theta} \\
			0 & & & & 1
		\end{array}\right),
		A_{p}=\left(\begin{array}{ccccc}
			0 & (-1)^k & & & \\
			& \tlam & (-1)^k\lambda & & \\
			& & \ddots & \ddots & \\
			& & & \tlam & (-1)^k\lambda \\
			(-1)^k\lambda & & & & \tlam
		\end{array}\right),
	\end{equation*}
	\begin{equation*}
		A_{q}=\left(\begin{array}{ccccc}
			\lambda & (-1)^k\tlam & & & \\
			& \lambda & (-1)^k\tlam & & \\
			& & \ddots & \ddots & \\
			& & & \lambda & (-1)^k\tlam \\
			0 & & & & 1
		\end{array}\right),
		A_{r}=\left(\begin{array}{ccccc}
			0 & (-1)^k & & & \\
			& \tilth & (-1)^k\theta & & \\
			& & \ddots & \ddots & \\
			& & & \tilth & (-1)^k\theta \\
			(-1)^k\theta & & & & \tilth
		\end{array}\right),
\end{equation*} }
due to the exact collocation at one of the boundary point in \eqref{wim}.
A superconvergent bound of interpolation errors analogous to \eqref{u with uli} in Lemma \ref{pro of interpolation functions} can thus be derived.

To be compatible with superconvergence property for mixed boundary conditions,  for $u_0 \in W^{k+\ell+4, \infty}(\ih)$ and $u_t(0) \in W^{k+\ell+1, \infty}(\ih)$ ($\ell\in \z{k}$), we modify the numerical initial condition \eqref{ini con} to
\begin{align}\label{inimixed}
	r_h(x,0) = \il{r}(x,0) = \tptt r_0(x) - \wl{r}(x,0),\quad r_0 = \partial_x^3 u_0(x),
\end{align}
with $ u_h(x,0)$, $p_h(x,0)$ and $q_h(x,0)$ still being the solutions to \eqref{im}, where $q_0 = \partial_x^2 u_0(x), p_0 = \partial_x u_0(x)$.

We are now ready to show superconvergence results of LDG methods using numerical fluxes \eqref{fluxm} for the case with mixed boundary conditions.
\begin{theorem}\label{scth mixed}
Suppose $u$ is the exact solution of \eqref{eq_lin_4_} with mixed boundary conditions \eqref{bd mixed}, which is sufficiently smooth, e.g., $u \in W^{2k+4,\infty}(\ih)$ and $u_t \in W^{2k+3,\infty}(\ih)$. Assume that $v_h \in V_h$ ($v=u,p,q,r$) are the solutions of LDG scheme with generalized numerical fluxes \eqref{fluxm} satisfying $ \theta=\lambda > \ff $.
Then, we have the following superconvergent results.

$\rrom 1$ Supercloseness between interpolation functions and LDG solutions:
\begin{equation}\label{sclose mixed}
	\norm{\be{u}(t)} + \norm{\be{p}(t)} + \leri{\inttdtau{ (\norms{\be{q}} + \norms{\be{r}}) }}^{\ff} \le \chkl,
\end{equation}
where $C$ depends on $\norm{u}_{k+\ell+4,\infty}$ and $\norm{u_t}_{k+\ell+3,\infty}$, but is independent of $h$.

$\rrom 2$ Superconvergence results for $ v = u,p$:

(1) Superconvergence of the numerical flux
\begin{align*}
	\norm{e_{vn}} = \Big(\frac{1}{N} \sum_{j =1}^{N}\mo{\left(v -\hat{v}_h\right)(\xjp ,t)}^2\Big)^{\ff} \le	Ch^{2k+1},
\end{align*}
where $C$ depends on $\norm{u}_{2k+4,\infty}$ and $\norm{u_t}_{2k+3,\infty}$, but is independent of $h$.

(2) Superconvergence of the cell averages
\begin{align*}
	\norm{e_v}_c  = \Big(\frac{1}{N} \sum_{j =1}^{N}\Big|{\frac{1}{h_j}\intijd{(v-v_h)(x,t)}}\Big|^2\Big)^{\ff} \le Ch^{2k+1},
\end{align*}
where $C$ depends on $\norm{u}_{2k+4,\infty}$ and $\norm{u_t}_{2k+3,\infty}$, but is independent of $h$.

(3) When $\ell \ge 2$, the function value approximations of $v_h~(v=u,p)$ are $(k+2)$th order superconvergent at generalized Radau points $\mathcal{R}_{j,m}^{\sigma}~(\sigma = \theta, \tlam)$, and the derivative value approximations are $(k+1)$th order superconvergent at generalized Radau points $\mathcal{R}_{j,m}^{\sigma,\star}$, i.e.,
\begin{align*}
	\norm{e_{vr}} =\max_{j\in \z{N}}\mo{(v - v_h)(\mathcal{R}_{j,m}^{\sigma})} \le	Ch^{k+2} ,\quad
	\norm{e^{\star}_{vr}}= \max_{j\in \z{N}}\mo{\partial_x(v - v_h)(\mathcal{R}_{j,m}^{\sigma,\star})} \le	Ch^{k+1},
\end{align*}
where $C$ depends on $\norm{u}_{k+\ell+4,\infty}$ and $\norm{u_t}_{k+\ell+3,\infty}$, but is independent of $h$.

(4) Supercloseness between the projection $ \tilde P_{\sigma}v~(\sigma = \theta, \tlam)$ and LDG solution $v_h$ ($ v = u,p $)
\begin{equation*}
	\norm{\tilde P_{\sigma}v - v_h} \le Ch^{k+2},
\end{equation*}
where $C$ depends on $\norm{u}_{k+5,\infty}$ and $\norm{u_t}_{k+4,\infty}$, but is independent of $h$.
\end{theorem}

\begin{proof}
Since the proof line is similar to that in Theorems \ref{scth}--\ref{scorderth}, we  need only to pay attention to the following modified generalized skew-symmetry property
\begin{align*}
		& \bb{H}^{\theta_1}(w,v) + \bb{H}^{\theta_2}(v,w) \\
      & \ind = 		(\tilde{\theta}_2 - \theta_1)\sum_{j=2}^{N-1}\jump{w}_{\jm}\jump{v}_{\jm}
		 + (\tilde{\theta}_2  - \theta_1) w^{+}v^+{|_{\frac{1}{2}}}
		 + \theta_1 w^- v^+{|_{\frac{1}{2}}}
		 + \theta_2 v^- w^+{|_{\frac{1}{2}}}\\
		& \ind \ind  +(\tilde{\theta}_2  - \theta_1) w^{-}v^-{|_{N+\frac{1}{2}}} - \tilde{\theta}_1 w^+ v^-{|_{N+\frac{1}{2}}}  - \tilde{\theta}_2 v^+ w^-{|_{N+\frac{1}{2}}},~w,v\in H^1(\ih),
\end{align*}
which is useful in proving supercloseness result \eqref{sclose mixed}.
More details are omitted.
\end{proof}

\section{Numerical experiments}\label{experiments}

Based on the idea of \cite{cao2017upwind-biased}, let us first describe the implementation of numerical initial condition, and for \eqref{eq_lin_4_} with periodic boundary condition, $u_h(\cdot, 0)$ can be chosen by the following procedure.
\begin{enumerate}
  \item [(1)] Let $r_0  = \partial_x^3 u_0$. Calculate $\ptt r_0$ by \eqref{ggr_pro} and $\wi{r}$, $i\in\z{k}$ by \eqref{wi_lin_4}--\eqref{wi initial}.
  \item [(2)] Set $r_h = \ptt r_0 - \wl{r}$, then $e_r = r -\ptt r_0 + \wl{r}$. By \eqref{orthj2} and \eqref{ggr_pro}, we get, for $\phi \in P^{k}(\ce)$ and $j\in \z{N}$,
\begin{align*}	
	(e_r,\phi) = - (\pl q - q_h,\phi_x) + \sumjn {(\pl q - q_h)}^{(\lambda)} \phi ^-{\ljp} - \sumjn {(\pl q - q_h)}^{(\lambda)} \phi^+{\ljm}.
	\end{align*}
Using integration by parts for $(e_r,\phi)$,
	\begin{align*}
		(e_r,\phi) = \leri{(\bhj \di e_r)_x,\phi} = -\leri{\bhj \di e_r, \phi_x}
	\end{align*}
implied by $\bhj \di e_r(\xjp^-)  = ( e_r,1)_j = 0$, $\bhj \di e_r(\xjm^+)  = 0$, and taking $ {(\pl q - q_h)}^{\la}_{j-\frac12} = 0$, we have
$$(\pl q - q_h,\phi_x) = \leri{\bhj \di e_r, \phi_x},$$
which can be used to determine $\pl q - q_h$.
  \item [(3)] Calculate $ e_q $, and as in (2), taking $ {(\ptl p - p_h)}^{\tla}_{j-\frac12} = 0 $, compute $\ptl p - p_h $ by $$(\ptl p - p_h,\psi_x) = \leri{\bhj \di e_q, \psi_x}.$$
\item [(4)] Calculate $ e_p $, and as in (2),  taking $ {(\pt u - u_h)}^{\thet}_{j-\frac12} = 0 $, compute $\pt u - u_h$ by $$(\pt u - u_h,\zeta_x) = \leri{\bhj \di e_p, \zeta_x}.$$
\item [(5)] Calculate $ u_h (0) = (\pt u - (\pt u - u_h))(0).$
\end{enumerate}

Next, we provide some numerical examples to support theoretical results. We adopt the above special initial solution and use the third-order explicit total variation diminishing Runge--Kutta method for time discretization with $\Delta t = CFL * h^4$, where $CFL=0.001$ for $P^1$, $CFL=0.0001$ for $P^2$, $CFL=0.00001$ for $P^3$.

\begin{example}\label{ex2}
Consider
\begin{equation}\label{eq2}
\left\{
\begin{array}{lr}
u_t + u_x + u_{xx} + u_{xxxx} = 0,\\
u(x,0) = \sin(x),
\end{array}
\right.
\end{equation}
with periodic boundary conditions. The exact solution is
\begin{equation}\label{exact}
u(x,t) = \sin(x-t).
\end{equation}
\end{example}

\begin{table}[ht]\centering
	\caption{Errors and orders for Example \ref{ex2} with $ \theta=0.8$, $\lambda=1.2$,  $ T=0.1 $, $k=1$.}	
	\setlength{\tabcolsep}{3.2mm}
	\resizebox{\textwidth}{!}
	{\begin{tabular}{ccccccccccc}
		\toprule
		$N$   & $\norm{e_{un}}$ & Order & $\norm{e_{u}}_c$ & Order & $\norm{e_{ur}}$& Order & $\norm{e^{\star}_{ur}}$ & Order                        & $\norm{\pt u - u_h} $   & Order \\
		\hline
		16  & 1.79E-04         & --  & 4.72E-04         & --  & 6.11E-04                   & --  & 1.76E-02                       & --                         & 1.50E-03 & --   \\
		32  & 1.89E-05         & 3.25  & 6.11E-05         & 2.95  & 6.89E-05                   & 3.15  & 4.50E-03                       & 1.96                         & 1.98E-04 & 2.95   \\
		64  & 2.09E-06         & 3.18  & 7.69E-06         & 2.99  & 8.07E-06                   & 3.09  & 1.10E-03                       & 1.99                         & 2.50E-05 & 2.99   \\
		128 & 2.43E-07         & 3.11  & 9.62E-07         & 3.00  & 9.78E-07                   & 3.04  & 2.85E-04                       & 2.00                         & 3.13E-06 & 3.00   \\
		\toprule
		$N$   & $\norm{e_{pn}}$ & Order & $\norm{e_{p}}_c$ & Order & $\norm{e_{pr}}$& Order & $\norm{e^{\star}_{pr}}$ & Order  & $\norm{\ptl p - p_h} $   & Order \\
		\hline
		16  & 5.16E-04         & --  & 1.78E-04         & --  & 1.20E-03                   & --  & 3.20E-03                       & --  & 6.07E-04 & --   \\
		32  & 6.34E-05         & 3.02  & 1.88E-05         & 3.24  & 1.56E-04                   & 2.99  & 7.78E-04                       & 2.02 & 7.11E-05 & 3.09   \\
		64  & 7.82E-06         & 3.02  & 2.09E-06         & 3.17  & 1.95E-05                   & 3.00  & 1.95E-04                       & 2.00                         & 8.54E-06 & 3.06   \\
		128 & 9.69E-07         & 3.01  & 2.43E-07         & 3.11  & 2.44E-06                   & 3.00  & 4.89E-05                       & 1.99                         & 1.05E-06 & 3.03   \\
		\toprule
		$N$   & $\norm{e_{qn}}$ & Order & $\norm{e_{q}}_c$ & Order & $\norm{e_{qr}}$& Order & $\norm{e^{\star}_{qr}}$ & Order                 & $\norm{\pl q - q_h} $   & Order \\
		\hline
		16  & 3.69E-04         & --  & 5.13E-04         & --  & 1.00E-03                   & --  & 4.80E-03                       & --                         & 1.50E-03 & --   \\
		32  & 5.52E-05         & 2.74  & 6.33E-05         & 3.02  & 1.38E-04                   & 2.91  & 1.40E-03                       & 1.80                         & 1.96E-04 & 2.95   \\
		64  & 7.50E-06         & 2.88  & 7.81E-06         & 3.02  & 1.76E-05                   & 2.97  & 3.63E-04                       & 1.93                         & 2.48E-05 & 2.98   \\
		128 & 9.76E-07         & 2.94  & 9.69E-07         & 3.01  & 2.21E-06                   & 2.99  & 9.28E-05                       & 1.97                         & 3.12E-06 & 2.99   \\
		\toprule
		$N$   & $\norm{e_{rn}}$ & Order & $\norm{e_{r}}_c$ & Order & $\norm{e_{rr}}$& Order & $\norm{e^{\star}_{rr}}$ & Order                  & $\norm{\ptt r - r_h} $   & Order \\
		\hline
		16  & 1.20E-03         & --  & 3.66E-04         & --  & 1.30E-03                   & --  & 2.55E-02                       & --                         & 3.30E-03 & --   \\
		32  & 1.60E-04         & 2.94  & 5.51E-05         & 2.73  & 1.47E-04                   & 3.10  & 6.40E-03                       & 2.00                         & 4.49E-04 & 2.89   \\
		64  & 2.03E-05         & 2.98  & 7.50E-06         & 2.88  & 1.74E-05                   & 3.07  & 1.60E-03                       & 2.01                         & 5.78E-05 & 2.96   \\
		128 & 2.56E-06         & 2.99  & 9.76E-07         & 2.94  & 2.12E-06                   & 3.04  & 3.93E-04                       & 2.01                         & 7.33E-06 & 2.98\\
		\bottomrule		
	\end{tabular}}\label{ex2k1b}
\end{table}

\begin{table}[ht]\centering
	\caption{Errors and orders for Example \ref{ex2} with $ \theta=0.8$, $\lambda=1.2$, $ T=0.1 $, $k=2,3$.}	
	{\begin{tabular}{cccccccccc}
			\toprule
			\multicolumn{1}{l}{} & $N$   & $\norm{e_{un}}$ & Order & $\norm{e_{u}}_c$ & Order & $\norm{e_{ur}}$& Order & $\norm{e^{\star}_{ur}}$ & Order \\
			\hline
			\multirow{4}{*}{$ P^2 $}  & 8    & 6.93E-06 & --      & 4.24E-05 & --      & 4.55E-04 & --      & 2.97E-03 & --                \\
			&16   & 1.01E-07 & 6.10      & 1.35E-06 & 4.97      & 3.10E-05 & 3.88      & 3.74E-04 & 2.99                \\
			&32   & 2.53E-09 & 5.31      & 4.25E-08 & 4.99      & 1.97E-06 & 3.98      & 4.75E-05 & 2.98                \\
			&64   & 7.99E-11 & 4.98      & 1.33E-09 & 5.00      & 1.24E-07 & 3.99      & 5.95E-06 & 3.00      \\
			\hline
			\multirow{4}{*}{$ P^3 $}  & 10   & 9.51E-09 &--      & 2.25E-08 &--      & 1.00E-05 &--      & 7.44E-04 &--       \\
			& 15   & 3.89E-10 & 7.89     & 1.54E-09 & 6.62     & 1.41E-06 & 4.83     & 1.56E-04 & 3.85      \\
			& 20   & 4.10E-11 & 7.82     & 2.21E-10 & 6.74     & 3.43E-07 & 4.91     & 5.05E-05 & 3.92      \\
			& 25   & 5.69E-12 & 8.85     & 4.81E-11 & 6.84     & 1.14E-07 & 4.95     & 2.09E-05 & 3.96      \\
			\toprule
	\end{tabular}}\label{ex2k123}
\end{table}

\begin{figure}[ht]\centering
	\centering
\includegraphics[scale=0.6]{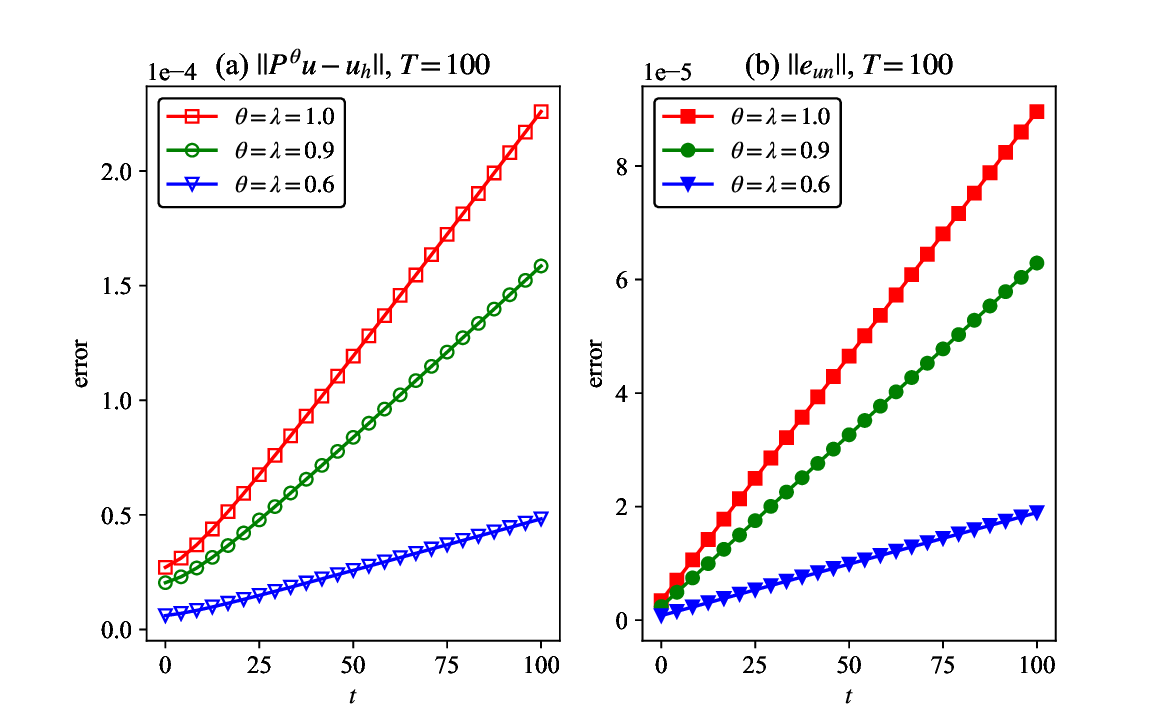}
	\caption{Time evolution of the error for Example \ref{ex2} with different weights, $ k=2 $, $ N=16$, $T=100$. }
	\label{fig:fig1label}
\end{figure}

The errors and orders for $u_h $, $p_h$, $ q_h$, and $r_h$ with generalized numerical fluxes and $1\le k \le 3$ are shown in Tables \ref{ex2k1b}--\ref{ex2k123}. We can see that
the errors of numerical fluxes and cell averages achieve $(2k+1)$th order, and the function value error achieve $(k+2)$th ($(k+1)$th) order at generalized (derivative) Radau points. Also, the error between GGR projection and numerical solution is of $ (k+2) $th order. This demonstrates that the results in Theorem \ref{scorderth} are valid. In addition, the time evolution of the error up to $T=100$ for Example \ref{ex2} is given in Figure \ref{fig:fig1label}, from which we can see that, at least for $\norm{\pt u -u_h}$ and $\norm{e_{un}}$, the generalized fluxes ($\theta = \lambda = 0.6 $ and $\theta =\lambda = 0.9$) produce a slower error growth when compared with upwind and alternating fluxes ($\theta =\lambda = 1$).

\begin{example}\label{exnonper}
In this example, consider the problem \eqref{eq2} with mixed boundary conditions
\begin{equation}\label{expm}
		u(0,t) = g_1(t),\quad u_x(2\pi,t) = g_2(t), \quad u_{xx}(0,t) = g_3(t), \quad u_{xxx}(2\pi,t) = g_4(t),
\end{equation}
where $g_i(t)$ ($i=1,2,3,4$) are suitably chosen such that the exact solution is \eqref{exact}.
\end{example}

\begin{table}[ht]\centering
	\caption{Errors and orders for Example \ref{exnonper} with mixed boundary conditions \eqref{expm} and $ \theta=1.1$, $\lambda=0.9$, $ T=0.1.$}
	{\begin{tabular}{cccccccccc}
			\toprule
			\multicolumn{1}{l}{} & $N$   & $\norm{e_{un}}$ & Order & $\norm{e_{u}}_c$ & Order & $\norm{e_{ur}}$& Order & $\norm{e^{\star}_{ur}}$ & Order \\
			\hline
				& 16  & 2.05E-04 & -- & 7.85E-04 & -- & 1.20E-03 & --  & 3.10E-03 & -- \\
				& 32  & 1.40E-05 & 3.87 & 9.44E-05 & 3.05 & 1.56E-04 & 2.98  & 8.36E-04 & 1.90 \\
				& 64  & 1.27E-06 & 3.47 & 1.17E-05 & 3.01 & 2.07E-05 & 2.91  & 2.13E-04 & 1.97 \\
				
				\multirow{-4}{*}{$P^1$} & 128 & 1.37E-07 & 3.21 & 1.46E-06 & 3.00 & 2.67E-06 & 2.96                        & 5.39E-05 & 1.98 \\
				\hline
				& 8   & 1.51E-05 & -- & 2.34E-05 & -- & 7.85E-05 & --                        & 2.80E-03 & -- \\
				& 16  & 1.90E-07 & 6.31 & 7.26E-07 & 5.01 & 7.40E-06 &  3.41 & 3.87E-04 & 2.88 \\
				& 32  & 3.23E-09 & 5.88 & 2.19E-08 & 5.05 & 5.00E-07 &  3.89 & 4.94E-05 & 2.97 \\
				
				\multirow{-4}{*}{$P^2$} & 64  & 1.03E-10 & 4.97 & 6.76E-10 & 5.02 & 4.15E-08 &  3.59 & 6.24E-06 & 2.99 \\
				\hline
				& 10  & 2.75E-08 & -- & 3.75E-08 & -- & 4.92E-06 & --                        & 6.68E-05 & -- \\
				& 15  & 1.63E-09 & 6.97 & 2.17E-09 & 7.02 & 6.41E-07 & 5.03                        & 1.33E-05 & 3.97 \\
				& 20  & 2.16E-10 & 7.02 & 2.87E-10 & 7.03 & 1.53E-07 & 4.99                        & 4.22E-06 & 4.00 \\
				\multirow{-4}{*}{$P^3$} & 25  & 4.56E-11 & 6.98 & 6.03E-11 & 6.99 & 5.05E-08 & 4.96                        & 1.74E-06 & 3.96   \\
			\toprule
		\end{tabular}\label{exmk123}}
\end{table}

Numerical errors and orders with generalized fluxes ($\theta=1.1$, $\lambda=0.9$) are provided in Table \ref{exmk123}, illustrating that the theoretical results in Theorem \ref{scth mixed} with mixed boundary conditions are true, even for  $\theta \neq \lambda$.

\begin{table}[ht]\centering
	\caption{Errors and orders for Example \ref{exnonper} with Dirichlet boundary conditions \eqref{expd} and $ \theta=1.1$, $\lambda=0.9$, $ T=0.1.$}	
	{\begin{tabular}{cccccccccc}
			\toprule
			\multicolumn{1}{l}{} & $N$   & $\norm{e_{un}}$ & Order & $\norm{e_{u}}_c$ & Order & $\norm{e_{ur}}$& Order & $\norm{e^{\star}_{ur}}$ & Order \\
			\hline
			\multirow{4}{*}{$P^1$} & 16  & 1.30E-03 & -- & 1.60E-03 & -- & 3.40E-03 & -- & 6.30E-03 & -- \\
			& 32  & 1.59E-04 & 3.00 & 2.02E-04 & 3.00 & 4.22E-04 & 3.01 & 1.20E-03 & 2.45 \\
			& 64  & 1.03E-05 & 3.94 & 1.73E-05 & 3.55 & 3.12E-05 & 3.76 & 2.24E-04 & 2.37 \\
			& 128 & 6.27E-07 & 4.04 & 1.71E-06 & 3.34 & 2.82E-06 & 3.47 & 5.44E-05 & 2.04 \\
			\hline
			\multirow{4}{*}{$P^2$} & 8   & 1.43E-05 & -- & 2.98E-05 & -- & 1.08E-04 & -- & 2.90E-03 & -- \\
			& 16  & 1.77E-07 & 6.33 & 7.73E-07 & 5.27 & 6.99E-06 & 3.95 & 3.87E-04 & 2.88 \\
			& 32  & 1.47E-09 & 6.92 & 2.25E-08 & 5.10 & 4.94E-07 & 3.82 & 4.94E-05 & 2.97 \\
			& 64  & 6.29E-11 & 4.55 & 6.90E-10 & 5.03 & 4.14E-08 & 3.58 & 6.24E-06 & 2.99 \\
			\hline
			\multirow{4}{*}{$P^3$} & 10  & 3.20E-06 & -- & 2.98E-06 & -- & 9.03E-06 & -- & 7.03E-05 & -- \\
			& 15  & 1.77E-07 & 7.13 & 1.70E-07 & 7.06 & 8.43E-07 & 5.85 & 1.34E-05 & 4.08 \\
			& 20  & 2.31E-08 & 7.09 & 2.24E-08 & 7.05 & 1.76E-07 & 5.45 & 4.24E-06 & 4.00 \\
			& 25  & 4.75E-09 & 7.08 & 4.65E-09 & 7.05 & 5.52E-08 & 5.19 & 1.74E-06 & 3.98  \\
			\toprule
	\end{tabular}}\label{exdirik123}
\end{table}

We also consider \eqref{eq2} with Dirichlet boundary conditions
\begin{equation}\label{expd}
		u(0,t) = h_1(t),\quad u(2\pi,t) = h_2(t), \quad u_{x}(0,t) = h_3(t), \quad u_{x}(2\pi,t) = h_4(t),
\end{equation}
where $h_i(t)$ ($i=1,2,3,4$) are suitably chosen such that the exact solution is \eqref{exact}.  The numerical fluxes are
 \begin{align}\label{fluxd}
	\leri{ \hat{u}_h, \hat{p}_h, \hat{q}_h, \hat{r}_h  }_{\jp}= \left\{
	\begin{aligned}
		&\big( h_1,                 h_3,    q^{+}_h + \kappa_1 \jump{p_h},              r^+_h                 \big)_{\ff},  ~&   &j=0, \\
		&\big( u^{\thet}_h,p^{\tla}_h,q^{\la}_h,r^{\tilde{\thet}}_h   \big)_{\jp},  ~&   &j=1,\cdots,N-1, \\
		&\big( h_2,    h_4,               q^-_h,  r^{-}_h - \kappa_2 \jump{u_h}                          \big)_{N+\ff},  ~&  &j=N,
	\end{aligned}
	\right.
\end{align}
where $\kappa_1 = \mathcal{O}\left(h^{-1}\right)$ and $\kappa_2 = \mathcal{O}\left(h^{-3 }\right)$ are penalty parameters. In Table \ref{exdirik123}, we show superconvergence results for generalized fluxes with $ \theta=1.1$, $ \lambda=0.9$ and $\kappa_1 = 10$, $\kappa_2 = 15$, indicating that the superconvergent results are also valid for Dirichlet boundary conditions.

\begin{example}\label{exd}
In the case of discontinuous initial value problem,  we consider
\begin{equation*}
		\left\{
		\begin{array}{lr}
			u_t + u_x + u_{xx} + u_{xxxx} = 0,\\
			u(x,0) = \left\{
			\begin{aligned}
				1, & ~&  &\mo{x}\le 0.5 ,\\
				0, & ~&  &\text{otherwise},
			\end{aligned}
			\right.
		\end{array}
		\right.
\end{equation*}
with periodic boundary conditions.
With negligible error, the exact solution is taken as that in \cite{meng2012superconvergence}, i.e.,
\begin{equation*}
		u(x,t) = \ff + 2\sum_{\omega=1}^{5}e^{(\omega^2\pi^2 - \omega^4\pi^4)t}\frac{\sin(\frac{\omega\pi}{2})}{\omega\pi}\cos(\omega\pi(x-t)).
\end{equation*}
\end{example}
Superconvergent orders of numerical fluxes, cell averages and Radau points with $ \theta=0.8$, $\lambda=1.2$, $T = 0.01$ are presented in Table \ref{exdk123}. This demonstrates that superconvergent results also hold for
discontinuous initial value problem.

\begin{table}[ht]\centering
	\caption{Errors and orders for Example \ref{exd} with a discontinuous initial value data and $\theta=0.8$, $\lambda=1.2$, $ T=0.01.$}	
	{\begin{tabular}{cccccccccc}
			\toprule
			\multicolumn{1}{l}{} & $N$   & $\norm{e_{un}}$ & Order & $\norm{e_{u}}_c$ & Order & $\norm{e_{ur}}$& Order & $\norm{e^{\star}_{ur}}$ & Order \\
			\hline
		\multirow{4}{*}{$P^1$} & 8   & 9.78E-04         &--   & 4.40E-04         &--   & 1.60E-03                   &--   & 4.78E-02                       &--   \\
		& 16  & 1.04E-04         & 3.24  & 2.77E-05         & 3.99  & 2.62E-04                   & 2.63  & 1.43E-02                       & 1.74  \\
		& 24  & 2.97E-05         & 3.08  & 5.82E-06         & 3.85  & 8.12E-05                   & 2.89  & 6.60E-03                       & 1.91  \\
		& 32  & 1.24E-05         & 3.04  & 2.08E-06         & 3.58  & 3.48E-05                   & 2.95  & 3.70E-03                       & 1.96  \\
		\hline
		\multirow{4}{*}{$P^2$} & 12  & 1.32E-06         &--   & 1.10E-07         &--   & 2.51E-05                   &--   & 7.39E-04                       &--   \\
		& 16  & 3.17E-07         & 4.97  & 2.41E-08         & 5.28  & 8.16E-06                   & 3.91  & 3.14E-04                       & 2.98  \\
		& 20  & 1.04E-07         & 4.98  & 7.55E-09         & 5.19  & 3.38E-06                   & 3.95  & 1.61E-04                       & 2.99  \\
		& 24  & 4.20E-08         & 4.99  & 2.95E-09         & 5.15  & 1.64E-06                   & 3.97  & 9.34E-05                       & 2.99  \\
		\hline
		\multirow{4}{*}{$P^3$} & 4   & 4.16E-06         &--   & 1.67E-06         &--   & 1.51E-04                   &--   & 1.61E-02                       &--   \\
		& 8   & 3.48E-08         & 6.90  & 7.88E-09         & 7.73  & 7.47E-06                   & 4.34  & 1.40E-03                       & 3.51  \\
		& 12  & 2.20E-09         & 6.81  & 3.31E-10         & 7.82  & 1.10E-06                   & 4.73  & 3.06E-04                       & 3.77  \\
		& 16  & 3.03E-10         & 6.89  & 3.54E-11         & 7.77  & 2.71E-07                   & 4.86  & 1.01E-04                       & 3.87    \\
			\toprule
	\end{tabular}}\label{exdk123}
\end{table}

\begin{example}\label{exn}
To investigate the case for nonlinear problems, consider the Kuramoto--Sivashinsky equation
\begin{equation*}
	u_t + {f(u)}_x + u_{xx} + \sigma u_{xxx} + u_{xxxx} = 0, ~x \in [-30,30]
\end{equation*}
with $ f(u) = \frac{u^2}{2} $ and the exact solution is
\begin{align*}
u(x,t) = c+ 9 -15\leri{\tanh\leri{k(x-ct-x_0)} + \tanh^{2}\leri{k(x-ct-x_0)} - \tanh^{3}\leri{k(x-ct-x_0)}},
\end{align*}
where $ \sigma = 4 $, $ c=6 $, $ k=\ff $ and $ x_0=-10 $. Note that periodic boundary conditions can be used, as the boundary value is quite small for short time simulations, say $T=0.1$.
\end{example}
	
We use the Godunov flux for the nonlinear convection term and generalized fluxes for linear terms. Table \ref{exnk123} lists superconvergent orders for numerical fluxes generalized fluxes with $ \theta=1.1$, $\lambda=0.9$, which shows that the superconvergence property also holds true for nonlinear problems.
\begin{table}[ht]\centering
		\caption{Errors and orders for Example \ref{exn} with $ \theta=1.1$, $\lambda=0.9$, $ T=0.1.$}	
		{\begin{tabular}{cccccccccc}
				\toprule
				\multicolumn{1}{l}{} & $N$   & $\norm{e_{un}}$ & Order & $\norm{e_{u}}_c$ & Order & $\norm{e_{ur}}$& Order & $\norm{e^{\star}_{ur}}$ & Order \\
				\hline
				\multirow{4}{*}{$P^1$} & 160 & 1.60E-03         &--  & 1.50E-03         &--  & 1.43E-02                   &--  & 4.14E-02                       &--  \\
				& 320 & 1.66E-04         & 3.27  & 1.97E-04         & 2.90  & 1.60E-03                   & 3.19  & 1.23E-02                       & 1.75  \\
				& 480 & 4.59E-05         & 3.16  & 6.02E-05         & 2.92  & 4.53E-04                   & 3.07  & 5.60E-03                       & 1.92  \\
				& 640 & 1.88E-05         & 3.11  & 2.58E-05         & 2.94  & 1.94E-04                   & 2.94  & 3.20E-03                       & 1.96  \\
				\hline
				\multirow{4}{*}{$P^2$} & 120 & 5.82E-05         &--  & 2.82E-05         &--  & 5.27E-04                   &--  & 2.13E-02                       &--  \\
				& 160 & 1.37E-05         & 5.02  & 6.97E-06         & 4.86  & 1.82E-04                   & 3.70  & 9.37E-03                       & 2.85  \\
				& 200 & 4.45E-06         & 5.06  & 2.29E-06         & 4.98  & 7.20E-05                   & 4.15  & 5.23E-03                       & 2.61  \\
				& 240 & 1.76E-06         & 5.07  & 9.13E-07         & 5.04  & 3.69E-05                   & 3.67  & 2.99E-03                       & 3.08  \\
				\hline
				\multirow{4}{*}{$P^3$} & 40  & 1.19E-03         &--  & 1.04E-03         &--  & 2.19E-02                   &--  & 7.96E-02                       &--  \\
				& 80  & 3.05E-06         & 8.61  & 5.35E-06         & 7.60  & 6.16E-04                   & 5.15  & 6.85E-03                       & 3.54  \\
				& 120 & 1.55E-07         & 7.34  & 2.28E-07         & 7.78  & 8.75E-05                   & 4.81  & 1.43E-03                       & 3.86  \\
				& 160 & 2.37E-08         & 6.53  & 3.08E-08         & 6.97  & 2.45E-05                   & 4.43  & 5.44E-04                       & 3.37      \\
				\toprule
\end{tabular}}\label{exnk123}
\end{table}

\section{Concluding remarks}\label{conclusion}
In this paper, we study superconvergence of the LDG method using generalized numerical fluxes for one-dimensional linear fourth-order problems.
By constructing correction functions and using properties of GGR projections, a superconvergent bound for interpolation errors is shown. Under a suitable numerical initial condition,
superconvergence regarding numerical flux, cell averages and generalized Radau points are established. Extension to mixed boundary conditions is given.
Problems with Dirichlet boundary conditions, discontinuous initial condition and nonlinear convection term are also numerically tested, demonstrating that the superconvergence results hold for more general cases.
Analysis of nonlinear and multidimensional equations is challenging, which constitutes our future work.

\bibliographystyle{amsplain}
\bibliography{paper}

\providecommand{\bysame}{\leavevmode\hbox to3em{\hrulefill}\thinspace}
\providecommand{\MR}{\relax\ifhmode\unskip\space\fi MR }
\providecommand{\MRhref}[2]{%
  \href{http://www.ams.org/mathscinet-getitem?mr=#1}{#2}
}
\providecommand{\href}[2]{#2}
\begin{thebibliography}{10}

\bibitem{brenner2003poincare}
Susanne~C Brenner, \emph{{Poincar{\'e}--Friedrichs Inequalities for Piecewise
  $H^1$ Functions}}, SIAM Journal on Numerical Analysis \textbf{41} (2003),
  no.~1, 306--324.

\bibitem{cahn1958hilliard}
John~W Cahn and John~E Hilliard, \emph{{Free energy of a nonuniform system. I.
  Interfacial free energy}}, The Journal of Chemical Physics \textbf{28}
  (1958), no.~2, 258--267.

\bibitem{cao2017highorder}
Waixiang Cao and Qiumei Huang, \emph{{Superconvergence of local discontinuous
  Galerkin methods for partial differential equations with higher order
  derivatives}}, Journal of Scientific Computing \textbf{72} (2017), no.~2,
  761--791.

\bibitem{cao2017upwind-biased}
Waixiang Cao, Dongfang Li, Yang Yang, and Zhimin Zhang, \emph{{Superconvergence
  of discontinuous Galerkin methods based on upwind-biased fluxes for 1D linear
  hyperbolic equations}}, ESAIM: Mathematical Modelling and Numerical Analysis
  \textbf{51} (2017), no.~2, 467--486.

\bibitem{cao2014SCDG}
Waixiang Cao, Zhimin Zhang, and Qingsong Zou, \emph{{Superconvergence of
  discontinuous Galerkin methods for linear hyperbolic equations}}, SIAM
  Journal on Numerical Analysis \textbf{52} (2014), no.~5, 2555--2573.

\bibitem{cheng2017application}
Yao Cheng, Xiong Meng, and Qiang Zhang, \emph{Application of generalized
  {G}auss--{R}adau projections for the local discontinuous {G}alerkin method
  for linear convection-diffusion equations}, Mathematics of Computation
  \textbf{86} (2017), no.~305, 1233--1267.

\bibitem{cockburn1990RKDG4}
Bernardo Cockburn, Suchung Hou, and Chi-Wang Shu, \emph{{The Runge--Kutta local
  projection discontinuous Galerkin finite element method for conservation laws
  IV: The multidimensional case}}, Mathematics of Computation \textbf{54}
  (1990), no.~190, 545--581.

\bibitem{cockburn1989TVB2}
Bernardo Cockburn and Chi-Wang Shu, \emph{{TVB Runge--Kutta local projection
  discontinuous Galerkin finite element method for conservation laws II:
  General framework}}, Mathematics of Computation \textbf{52} (1989), no.~186,
  411--435.

\bibitem{cockburn1998LDG}
\bysame, \emph{{The local discontinuous Galerkin method for time-dependent
  convection-diffusion systems}}, SIAM Journal on Numerical Analysis
  \textbf{35} (1998), no.~6, 2440--2463.

\bibitem{cockburn1998RKDG5}
\bysame, \emph{{The Runge--Kutta discontinuous Galerkin method for conservation
  laws V: multidimensional systems}}, Journal of Computational Physics
  \textbf{141} (1998), no.~2, 199--224.

\bibitem{dong2009analysis}
Bo~Dong and Chi-Wang Shu, \emph{{Analysis of a local discontinuous Galerkin
  method for linear time-dependent fourth-order problems}}, SIAM Journal on
  Numerical Analysis \textbf{47} (2009), no.~5, 3240--3268.

\bibitem{hutridurga2023vvb}
Harsha Hutridurga, Krishan Kumar, and Amiya~K Pani, \emph{{Discontinuous
  Galerkin methods with generalized numerical fluxes for the Vlasov-Viscous
  Burgers’ system}}, Journal of Scientific Computing \textbf{96} (2023),
  no.~1, article no.7.

\bibitem{lijia2020analysisKdV}
Jia Li, Dazhi Zhang, Xiong Meng, and Boying Wu, \emph{{Analysis of local
  discontinuous Galerkin methods with generalized numerical fluxes for
  linearized KdV equations}}, Mathematics of Computation \textbf{89} (2020),
  no.~325, 2085--2111.

\bibitem{liu2015local}
Hailiang Liu and Nattapol Ploymaklam, \emph{A local discontinuous galerkin
  method for the {B}urgers--{P}oisson equation}, Numerische Mathematik
  \textbf{129} (2015), no.~2, 321--351.

\bibitem{liu2021sc}
Xiaobin Liu, Dazhi Zhang, Xiong Meng, and Boying Wu, \emph{{Superconvergence of
  local discontinuous Galerkin methods with generalized alternating fluxes for
  1D linear convection-diffusion equations}}, Science China Mathematics
  \textbf{64} (2021), no.~6, 1305--1320.

\bibitem{liuyong2020SCofUWLDG4}
Yong Liu, Qi~Tao, and Chi-Wang Shu, \emph{{Analysis of optimal superconvergence
  of an ultraweak-local discontinuous Galerkin method for a time dependent
  fourth-order equation}}, ESAIM: Mathematical Modelling and Numerical Analysis
  \textbf{54} (2020), no.~6, 1797--1820.

\bibitem{meng2012superconvergence}
Xiong Meng, Chi-Wang Shu, and Boying Wu, \emph{{Superconvergence of the local
  discontinuous Galerkin method for linear fourth-order time-dependent problems
  in one space dimension}}, IMA Journal of Numerical Analysis \textbf{32}
  (2012), no.~4, 1294--1328.

\bibitem{meng2016optimal}
\bysame, \emph{{Optimal error estimates for discontinuous Galerkin methods
  based on upwind-biased fluxes for linear hyperbolic equations}}, Mathematics
  of Computation \textbf{85} (2016), no.~299, 1225--1261.

\bibitem{zhong2022sc}
Yuqing Miao, Jue Yan, and Xinghui Zhong, \emph{{Superconvergence study of the
  direct discontinuous Galerkin method and its variations for diffusion
  equations}}, Communications on Applied Mathematics and Computation \textbf{4}
  (2022), no.~1, 180--204.

\bibitem{ntoukas2021ch}
Gerasimos Ntoukas, Juan Manzanero, Gonzalo Rubio, Eusebio Valero, and Esteban
  Ferrer, \emph{{A free-energy stable p-adaptive nodal discontinuous Galerkin
  for the Cahn--Hilliard equation}}, Journal of Computational Physics
  \textbf{442} (2021), 110409.

\bibitem{shu2016discontinuous}
Chi-Wang Shu, \emph{Discontinuous galerkin methods for time-dependent
  convection dominated problems: Basics, recent developments and comparison
  with other methods}, Building bridges: connections and challenges in modern
  approaches to numerical partial differential equations (2016), 369--397.

\bibitem{wang2019implicit}
Haijin Wang, Qiang Zhang, and Chi-Wang Shu, \emph{{Implicit--explicit local
  discontinuous Galerkin methods with generalized alternating numerical fluxes
  for convection--diffusion problems}}, Journal of Scientific Computing
  \textbf{81} (2019), 2080--2114.

\bibitem{xu2005local2}
Yan Xu and Chi-Wang Shu, \emph{Local discontinuous galerkin methods for
  nonlinear {S}chr{\"o}dinger equations}, Journal of Computational Physics
  \textbf{205} (2005), no.~1, 72--97.

\bibitem{xu2005local}
\bysame, \emph{Local discontinuous {G}alerkin methods for two classes of
  two-dimensional nonlinear wave equations}, Physica D: Nonlinear Phenomena
  \textbf{208} (2005), no.~1-2, 21--58.

\bibitem{yan2002local}
Jue Yan and Chi-Wang Shu, \emph{A local discontinuous {G}alerkin method for
  {KdV} type equations}, SIAM Journal on Numerical Analysis \textbf{40} (2002),
  no.~2, 769--791.

\end{thebibliography}
\end{document}